\documentclass[12pt]{amsart}

\usepackage{subfig}
\usepackage[T1]{fontenc}
\usepackage[utf8]{inputenc}
\usepackage{fourier}
\usepackage{caption}
\usepackage[scaled=0.875]{helvet} 
 
\usepackage{amsmath,amssymb}
\usepackage[normalem]{ulem}
\usepackage{fancybox,graphicx}
\usepackage{enumerate}
\usepackage{eucal}
\usepackage{graphicx}
\usepackage{tabularx}
\usepackage{ulem}
\usepackage{xcolor}
\usepackage{dcolumn}
\usepackage{textcomp}
\usepackage{enumitem}
\usepackage{diagbox}
\usepackage{tabularx}
\usepackage{pgfplots}
\usepackage{mathrsfs}
\usetikzlibrary{arrows}
\pgfplotsset{compat=1.15}
\usepackage{lscape}
\usepackage{variations}
\usepackage{fancyhdr}
\usepackage{hyperref}
\usepackage[english]{babel}
\usepackage[np]{numprint}
\usepackage{booktabs}
\usepackage{amsthm}
\usepackage{amsmath}

\usetikzlibrary[patterns]

\usepackage{pgf,tikz}
\usetikzlibrary {shapes.geometric, calc}

\DeclareMathOperator\supp{supp}

\newcommand{\cc}{\mathcal{C}}

\newcommand{\s}{\mathcal{S}}

\newcommand{\R}{\mathbb{R}}

\newcommand{\Z}{\mathbb{Z}}

\newcommand{\C}{\mathbb{C}}

\newcommand{\norm}[1]{{\left\|{#1}\right\|}}

\newcommand{\abs}[1]{{\left|{#1}\right|}}
\newcommand{\scal}[1]{{\left\langle{#1}\right\rangle}}

\newcommand{\interior}[1]{%
  {\kern0pt#1}^{\mathrm{o}}%
}

\newcommand{\dst}{\displaystyle}

\newcommand{\ffi}{\varphi}
\newcommand{\eps}{\varepsilon}
\definecolor{qqqqff}{rgb}{0.,0.,1.}
\definecolor{ffqqqq}{rgb}{1.,0.,0.}
\definecolor{rvwvcq}{rgb}{0.08235294117647059,0.396078431372549,0.7529411764705882}
\definecolor{xfqqff}{rgb}{0.4980392156862745,0.,1.}
\definecolor{wwccqq}{rgb}{0.4,0.8,0.}
\definecolor{uuuuuu}{rgb}{0.26666666666666666,0.26666666666666666,0.26666666666666666}

\newtheorem{theorem}{Theorem}[section]
\newtheorem{definition}[theorem]{Definition}
\newtheorem{lemma}[theorem]{Lemma}

\newtheorem{proposition}[theorem]{Proposition}

\newtheorem*{theorem*}{Theorem}
\theoremstyle{remark}

\counterwithin{equation}{section}
\begin{document}

\title[Three Balls Inequality for Discrete Schrödinger Operator]{On the Three Balls Inequality for Discrete Schrödinger Operators on Certain Periodic Graphs}

\author{Yann Bourroux}
\address{Facultad de Ciencia y Tecnolog\'ia, Universidad del Pa\'is Vasco /Euskal Herriko Unibertsitatea (UPV/EHU), Departamento de Matem\'aticas, UPV/EHU, Apartado 644, 48080 Bilbao, Spain\\ \& Univ. Bordeaux, CNRS, Bordeaux INP, IMB, UMR 5251, F-33400 Talence, France}
\email{yann.bourroux@math.u-bordeaux.fr}

\author{Aingeru Fern\'andez-Bertolin}
\address{ Facultad de Ciencia y Tecnolog\'ia, Universidad del Pa\'is Vasco /Euskal Herriko Unibertsitatea (UPV/EHU), Departamento de Matem\'aticas, UPV/EHU, Apartado 644, 48080 Bilbao, Spain}
\email{aingeru.fernandez@ehu.eus}

\author{Philippe Jaming}
\address{Univ. Bordeaux, CNRS, Bordeaux INP, IMB, UMR 5251, F-33400 Talence, France}
\email{philippe.jaming@math.u-bordeaux.fr}






\begin{abstract}
We investigate quantitative unique continuation properties for discrete magnetic Schrödinger operators in certain periodic graphs. This unique continuation property will be quantified through what is known in the literature as a Three Balls Inequality. We are able to extend this inequality to another family of periodic graph which contains the Hexagonal lattice. We also give a sketch of the proof for general star periodic graph.
Our proofs are based on Carleman estimates.
\end{abstract}

\maketitle

\tableofcontents

\section{Introduction}

The aim of this article is to prove a Carleman inequality for discrete magnetic Schrödinger operators on a class of periodic graphs
beyond the square lattice. As an application we
obtain a quantitative uni\-que continuation property for discrete harmonic functions on those graphs, where the quantification is given through a so-called Three Balls Inequality. This inequality is then extended to a larger class of periodic graphs that also includes the hexagonal lattice.

Generally speaking, such an inequality has the following form: consider a solution $f$ of a PDE or a difference equation written as $Pf=0$ over some set $\Omega$.
We assume that $\Omega$ is endowed with some natural distance and denote by $B_r$ a ball of radius $r$
and $\|\cdot\|_{B_r}$ a norm of the restriction of $f$ to this ball. For instance, one may consider the $L^\infty$ norm
over the ball, $\|f\|_{B_r}=\sup_{x\in B_r}|f(x)|$. A Three Balls Inequality is then a way
to bound the norm of $f$ on some ball, say $B_1$, by the norm of $f$ on a smaller and on a larger ball, say $B_{1/2}$
and $B_2$. A typical inequality of that type would then take the form
\begin{eqnarray}
    \label{eq:3ballintro}
\|f\|_{B_1}\leq C_\alpha\|f\|_{B_{1/2}}^\alpha \|f\|_{B_2}^{1-\alpha}
\end{eqnarray}
for some $\alpha\in(0,1)$. We can also see this property as a propagation of smallness property, 
since if $\|f\|_{B_{1/2}}$ is small, say $\leq \eps$ and if we have a global bound of $f$ on $B_2$,
say $\|f\|_{B_2}\leq M$ then $f$ stays relatively small on $B_1$ since 
\eqref{eq:3ballintro} implies that $\|f\|_{B_1}\leq C_\alpha M^{1-\alpha}\eps^\alpha$.

Such inequalities are known for some time. For instance, if we consider $P=\Delta$ the Laplacian on $\R^d$, the Three Balls Inequality \eqref{eq:3ballintro} has been established by J. Korevaar and J.\,L.\,H. Meyers in \cite{korevaar_logarithmic_1994}
for harmonic functions $\Delta f=0$, both when $\|f\|_{B_r}$ is the $L^\infty$-norm 
$\|f\|_{B_r}=\sup_{B_r}|f|$ and also in the $L^2$-case when $\dst\|f\|_{B_r}=\left(\int_{B_r}|f(x)|^2\,\mbox{d}x\right)^{1/2}$.
It is also well known that this unique continuation property is satisfied when $\Delta$
is replaced by the Schrödinger operator $\Delta+B\cdot \nabla+V$, for certain classes of potentials $B$ and $V$, see for example \cite{kurata_unique_1997}. 

\smallskip

In the discrete setting, the situation is more complicated. Let us start with the simplest case
of the square lattice $\Z^d$ and the naturally associated graph Laplacian $\Delta_d$. 
In this case, Three Balls Inequalities in the form \eqref{eq:3ballintro} do not hold.
This situation is a bit awkward as, if one considers the square lattice $(h\mathbb{Z})^d$ with a mesh of size $h$
instead of $\Z^d$, then the graph Laplacian is a natural discretization of the Euclidean Laplacian in the sense
that, if $f$ is a $\cc^2$ function on $\R^d$ then $h^{-2}\Delta_{h\Z^d}f\to\Delta f $ when $h\to 0$.
One thus expects that the Three Balls Inequality should hold at least up to an error term when $h$ is small enough.
This has indeed been proved by M. Guadie and E. Malinnikova \cite{guadie_three_2014}: {\em there
exist $C,\delta>0$, $0<\alpha<1$ and $h_0$
such that if $0<h<h_0$, then for every $\cc^2$ function $f$ on $\R^d$ whose restriction to $h\Z^d$
satisfies $\Delta_{h\Z^d}f=0$, we have}
\begin{eqnarray}
    \label{eq:3ballintrocor}
\|f\|_{L^\infty(B_1)}\leq C\|f\|_{L^\infty(B_{1/2})}^\alpha \|f\|_{L^\infty(B_2)}^{1-\alpha}
+e^{-\delta/\sqrt{h}}\|f\|_{L^\infty(B_2)}.
\end{eqnarray}
Note that the error term $e^{-\delta/\sqrt{h}}\|f\|_{L^\infty(B_2)}$ goes to zero as the mesh size $h\to0$
{\it i.e.} when approaching the continuous regime. This error term is also shown to be optimal.
The proof relied on an explicit formula for the Poisson kernel of a discrete harmonic function
and is thus difficult to extend to more general settings, for instance when adding a potential to the Laplacian.

This difficulty has been overcome by the second author together with L. Roncal, A. R\"uland \& D. Stan 
\cite{fernandez-bertolin_discrete_2021} who took a different approach. They first establish a Carleman estimate
and then deduce the Three Balls Inequality from it. In particular, this allowed them to establish
\eqref{eq:3ballintrocor} but with $L^2$-norms instead of $L^\infty$-norms.

\smallskip

Our aim here is to extend this result to a larger class of graphs. The graphs we consider are known as
periodic graphs. Roughly speaking (precise definitions will be given in the next section), they consist of a finite graph whose vertices are vectors in $\R^d$
and this graph is then translated by Lattice elements in such a way that the edges are periodic.
The simplest such graphs are the lattices $A\Z^d$ with $A$ invertible and the edges are just the nearest neighbors. The hexagonal lattice is another
familiar example in this class.

Such graphs are key to understanding a wide range of applications including network design or physical systems like molecules and crystals.
To be more precise, structures similar to periodic graphs play a key role as candidates for the fundamental microscopic building block of space, see \cite{konopka_quantum_2008} or \cite{hamma_quantum_2010}. Crystal material property prediction is also important in modern physic as a way to discover new materials with good properties and it happens that crystals are studied as periodic graphs \cite{yan_periodic_2022}.

Mathematically, the set of vertices $\mathcal{V}$ and the edges $\mathcal{E}$ of a periodic graph $\Gamma$ are subsets of $\R^d$. Therefore, there is a natural scaling that we denote by
$h\Gamma$ and call $h$ the mesh size. We then consider the following operator,
\begin{equation}
    H_h f(x)=h^{-2}\Delta_{h\Gamma} f(x) + V(x)f(x) + h^{-1}\sum_{j=1}^k B(x)Df(x),
\end{equation}
where $f: h\mathcal{V} \mapsto \R$, $\Delta_{h\Gamma}$ is the discrete Laplace operator on the periodic graph $h\Gamma$ (see the definition below), $Df(x)$ is a difference operator that depends on the periodic graph $\Gamma$ (similar to the gradient in the continuous setting), $V: h\mathcal{V} \mapsto \R$ is a uniformly bounded potential and $B:h\mathcal{V} \mapsto \R^d$ is a uniformly bounded tensor field. 
Schrödinger operators on periodic discrete graphs have attracted a lot of attention due to their applications to the
study of electronic properties of crystalline structures, {\it see e.g.} \cite{NGMJZDGF04},
the book \cite{Ha02} and the survey \cite{CNGPNG09}. There is a vast mathematical literature as well, focusing mainly on spectral properties of those operators, {\it see e.g.}\cite{ando_spectral_2015,ando_inverse_2018,badanin_laplacians_2013,badanin_laplacians_2013,KS,Sa24}.

Another motivation behind the study of the operator $H_h$ on periodic graphs comes from the fact that, for certain graphs, 
the discrete Laplacian $\Delta_{h\Gamma}$ tends to approximate the continuous Laplacian $\Delta$ when the mesh size $h$ tends to $0$. For example, the Laplacian of the Square lattice is the usual discretization of the Laplacian in $\Z^d$.
Also, on the hexagonal lattice $\mathcal{H}$, if $g$ is a polynomial of degree 2 in two variables,
$$
g(x,y)=ax^2+bxy+cy^2+dx+ey+f,
$$
then
\begin{equation*}
    h^{-2}\Delta_\mathcal{H} g(xh,yh)=\frac{3}{2}\Delta g(x,y).
\end{equation*}
Hence the approximation error done, by approximating the continuum with the Hexagonal lattice, is of order $h^3$. 

In the free case, $B=0$, $V=0$, A. Bou-Rabee, W. Cooperman and S. Ganguly \cite{bou-rabee_unique_2023}
established the Three Balls Inequality in the spirit of \eqref{eq:3ballintrocor} for all planar periodic graphs by 
removing the need of computing the Poisson kernel. However, their result is in the $L^\infty$ norm, not in the $L^2$ one
and further seems not to extend to the more general cases $B\not=0$ or $V\not=0$ as the proof relies essentially on
a geometric interpretation of the Laplacian that is lacking when adding a potential.

We will overcome this by adapting the strategy of \cite{fernandez-bertolin_discrete_2021} to the larger setting of periodic graphs.
The first step consists in proving a discrete Carleman Inequality (Theorem \ref{Carleman:1}).
The Carleman weight we will use is an adaptation
of the weight used in \cite{koch_variable_2016,fernandez-bertolin_discrete_2021} where we replace the Euclidean norm with a norm $|.|_\Gamma$ adapted to the graph $\Gamma$. Also note that, contrary to the continuous case, the Carleman parameter $\tau$ is bounded in such a way that $\tau\lesssim h^{-1}$. When the mesh size tends to $0$, we get closer and closer to the continuous case. This condition is the key behind the error term that appears in discrete Three Balls Inequalities.
This strategy works in all dimensions but
unfortunately not for all graphs. The key difficulty
is that at one stage of the proof, we need to go to the Fourier side, reformulating one step in terms of a Fourier multiplier that one needs to bound from below.
There is a particular family of graphs that we call {\em 1-point periodic graphs}
for which we obtain such lower bounds as the
Fourier multiplier we have to deal with is complex valued. In the general case, this Fourier multiplier
is vector valued and lower bounds seem not to be available. This explains the restriction on the graphs we have to impose to obtain a Carleman estimate.

This Carleman estimate is then the main ingredient we need in order to prove Three Balls Inequalities. Indeed, using a cut-off argument and a Caccioppoli type estimate, we are able to derive Three Balls Inequality for the operator $H_h$ on 1-point periodic graph. In a last step, we are
able to extend this result to a larger family of periodic graphs, which contains the Hexagonal lattice. This is possible since the graphs $\Gamma$ in this family are such that we can construct
another graph $\Gamma'$ from a subset of its vertices which is a 1-point periodic graph and for which the restriction of the initial function $f$ to $\Gamma'$ satisfies the Three Balls Inequalities.
Moreover, the $\ell^2$ norm of $f$ on $\Gamma$ is equivalent to its $\ell^2$ norm on $\Gamma'$.

Note that, while this allows to significantly enlarge the class of periodic graphs for which Three Balls Inequalities are now available, 
it is still open whether such an inequality is valid
on all periodic graphs, most notably,
we would like to establish a unique continuation
property for the Laplacian on the Kagome Lattice,
{\it see} Figure \ref{kag} and \cite[Figure 4]{ando_spectral_2015} for further precision.

\smallskip

The paper is organized as follows: In Section 2, we will introduce the setting of our study. Section 3 is devoted to our discrete Carleman estimate (Theorem \ref{Carleman:1}). In Section 4, we give the statement of our Three Balls Inequality (Theorem \ref{TBT:4.2}) for the discrete magnetic Schrödinger operator on 1-point periodic graph together with its proof using the discrete Carleman estimate. Finally in the last Section, we extend the Three Balls Inequality to a larger class of periodic graphs, starting with the Hexagonal lattice.

\section{Setting and notation}
Throughout the paper, $h>0$ will be a parameter called {\em the mesh size} and that will be required to be sufficiently small.
\begin{definition}\label{definition:1}
    We say that $h\Gamma=\{h\mathcal{L},h\mathcal{V},h\mathcal{E}\}$ is a {\em periodic graph} if it satisfies the following properties.
\begin{enumerate}
    \item There is a basis $\ \mathbf{v}_1,\dots,\mathbf{v}_d$ of $\R^d$ such that
    $$
    \mathcal{L}=\left\{\mathbf{v}(n)=\sum_{j=1}^d\mathbf{v}_jn_j \ | \ n\in\Z^d\right\}
    $$
    the lattice generated by those vectors.
    \item There are $p_1,\dots,p_s$, $s$ points in $\R^d$, such that for every $i,j\in\{1,\dots,s\}$ with $i\not=j$, $p_i-p_j\notin\mathcal{L}$. The vertex set is then defined as 
    \begin{equation*}
        h\mathcal{V}=\bigcup_{i=1}^s h\mathcal{V}_i \mbox{ where } h\mathcal{V}_i=hp_i+h\mathcal{L}.
    \end{equation*}
    Without loss of generality, we will always assume that $p_1=0$.
    \item The edge set $h\mathcal{E}\subset h\mathcal{V}\times h\mathcal{V}$ satisfies the following periodicity condition: let $x,y\in h\mathcal{V}$ with $x\sim y$ {\it i.e.} $(x,y)$ is an edge, $(x,y)\in h\mathcal{E}$. Then, for every $m\in\Z^d$, $x+h\mathbf{v}(m)\sim y+h\mathbf{v}(m)$
    {\it i.e.} $\left(x+h\mathbf{v}(m),y+h\mathbf{v}(m)\right)$ is also an edge.
    \item The graph $h\Gamma$ is connected.
\end{enumerate} 
If convenient, we will say that $h\Gamma$ is
a {\em $d$-dimensional $s$-point} periodic graph.
\end{definition}

We will particularly focus on {\em 1-point periodic graphs} in which case $h\mathcal{V}=h\mathcal{L}$,
in particular, our Carleman estimate will only be valid in this case. 

\begin{center}
\begin{figure}[ht]
	\setlength{\unitlength}{0.7cm}
\begin{tabular}{ccc}
\begin{picture}(4.45,4.45)
\linethickness{0.05mm}
\multiput(0.45,0.45)(1,0){5}{\line(0,1){4}} 
\multiput(0.45,0.45)(0,1){5}{\line(1,0){4}} 
\put(0.15,0.15){$p_1$}
\put(0.45,0.45){\vector(1,0){1}}
\put(0.45,0.45){\vector(0,1){1}}
\put(0,0.8){$\mathbf{v}_2$}
\put(0.75,0.1){$\mathbf{v}_1$}
\multiput(0.45,0.45)(0,1){5}{\circle*{0.1}}
\multiput(1.45,0.45)(0,1){5}{\circle*{0.1}}
\multiput(2.45,0.45)(0,1){5}{\circle*{0.1}}
\multiput(3.45,0.45)(0,1){5}{\circle*{0.1}}
\multiput(4.45,0.45)(0,1){5}{\circle*{0.1}}
\end{picture}
&
\begin{picture}(4.45,4.45)
\linethickness{0.05mm}
\multiput(0.45,0.45)(1,0){5}{\line(0,1){4}} 
\multiput(0.45,0.45)(0,1){5}{\line(1,0){4}} 
\multiput(0.45,0.45)(0,1){5}{\circle*{0.1}}
\multiput(1.45,0.45)(0,1){5}{\circle*{0.1}}
\multiput(2.45,0.45)(0,1){5}{\circle*{0.1}}
\multiput(3.45,0.45)(0,1){5}{\circle*{0.1}}
\multiput(4.45,0.45)(0,1){5}{\circle*{0.1}}
\multiput(1.45,0.45)(1,0){4}{\line(-1,1){1}} 
\multiput(1.45,1.45)(1,0){4}{\line(-1,1){1}} 
\multiput(1.45,2.45)(1,0){4}{\line(-1,1){1}} 
\multiput(1.45,3.45)(1,0){4}{\line(-1,1){1}} 
\linethickness{0.5mm}
\put(0.45,0.45){\color{blue}\line(0,1){1}}
\put(0.45,0.45){\color{blue}\line(1,0){1}}
\put(0.45,1.45){\color{blue}\line(1,-1){1}}
\end{picture}
&
\begin{picture}(4.45,4.45)
\linethickness{0.05mm}
\multiput(0.45,0.45)(1,0){5}{\line(0,1){4}} 
\multiput(0.45,0.45)(0,1){5}{\line(1,0){4}} 
\multiput(0.45,0.45)(0,1){5}{\circle*{0.1}}
\multiput(1.45,0.45)(0,1){5}{\circle*{0.1}}
\multiput(2.45,0.45)(0,1){5}{\circle*{0.1}}
\multiput(3.45,0.45)(0,1){5}{\circle*{0.1}}
\multiput(4.45,0.45)(0,1){5}{\circle*{0.1}}
\multiput(0.45,0.45)(1,0){3}{\line(2,1){2}} 
\multiput(0.45,1.45)(1,0){3}{\line(2,1){2}} 
\multiput(0.45,2.45)(1,0){3}{\line(2,1){2}} 
\multiput(0.45,3.45)(1,0){3}{\line(2,1){2}} 
\multiput(0.45,0.95)(0,1){4}{\line(2,1){1}} 
\multiput(3.45,0.45)(0,1){4}{\line(2,1){1}} 
\multiput(0.45,0.45)(1,0){4}{\line(1,2){1}} 
\multiput(0.45,1.45)(1,0){4}{\line(1,2){1}} 
\multiput(0.45,2.45)(1,0){4}{\line(1,2){1}} 
\multiput(0.45,3.45)(1,0){4}{\line(1,2){0.5}} 
\multiput(0.95,0.45)(1,0){4}{\line(1,2){0.5}} 
\linethickness{0.5mm}
\put(0.45,0.45){\color{blue}\line(0,1){1}}
\put(0.45,0.45){\color{blue}\line(1,0){1}}
\put(0.45,0.45){\color{blue}\line(2,1){2}}
\put(0.45,0.45){\color{blue}\line(1,2){1}}
\end{picture}\\
a) Square lattice&b) triangular graph&c) another 1-point\\ 
&&periodic graph 
\end{tabular}	
\caption{Three 2-dimensional one point periodic graphs. Dots are the vertices and lines joining them the edges. The thick (blue) lines are the edges that are reproduced periodically}
\end{figure}
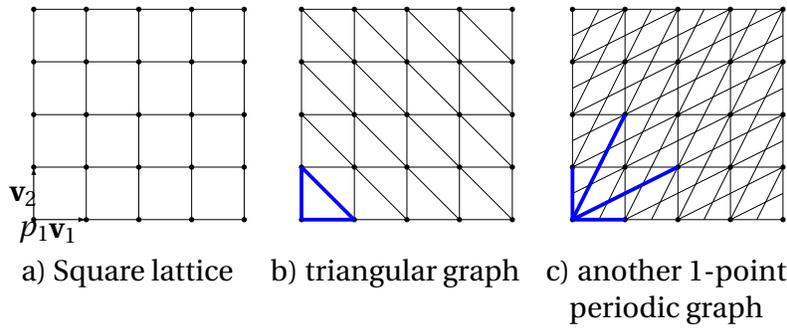
\end{center}

The Three Balls Inequality will be extended to more general $s$-point periodic graphs, including the Hexagonal lattice but {\em not} the Kagome lattice.

\begin{center}
    \tikzset{
    my hex/.style={regular polygon, regular polygon sides=6, draw, inner sep=0pt, outer sep=0pt, minimum size=1cm},
    my circ/.style={draw, circle, color=blue, fill=blue, inner sep=0pt, minimum size=0.75mm},
    my tri/.style={regular polygon, regular polygon sides=3, draw, inner sep=0pt, outer sep=0pt, minimum size=1mm, color=red, fill=red},
    my squa/.style={regular polygon, regular polygon sides=4, draw, color=green, fill=green, inner sep=0pt, minimum size=0.75mm},
    my tri2/.style={regular polygon, regular polygon sides=6, draw, inner sep=0pt, outer sep=0pt, minimum size=1mm, color=orange, fill=orange},
    }
\begin{figure}
\begin{tabular}{cc}
\begin{tikzpicture}[scale=0.35]
    \foreach \i in {0,...,3} 
    \foreach \j in {0,...,2} {
    \draw [line width=.5pt] (\i*2.414,\j*2.414) -- (\i*2.414+1,\j*2.414);
    \draw [line width=.5pt] (\i*2.414+1,\j*2.414) -- (\i*2.414+1.71,\j*2.414+0.71);
    \draw [line width=.5pt] (\i*2.414+1.71,\j*2.414+0.71) -- (\i*2.414+1.71,\j*2.414+1.71);
    \draw [line width=.5pt] (\i*2.414+1.71,\j*2.414+1.71) -- (\i*2.414+1,\j*2.414+2.414);
    \draw [line width=.5pt] (\i*2.414+1,\j*2.414+2.414) -- (\i*2.414,\j*2.414+2.414);
    \draw [line width=.5pt] (\i*2.414,\j*2.414+2.414) -- (\i*2.414-0.71,\j*2.414+1.71);
    \draw [line width=.5pt] (\i*2.414-0.71,\j*2.414+1.71) -- (\i*2.414-0.71,\j*2.414+0.71);
    \draw [line width=.5pt] (\i*2.414-0.71,\j*2.414+0.71) -- (\i*2.414,\j*2.414);
    }
    \foreach \i in {0,...,3} 
    \foreach \j in {0,...,3} {
    \node(h)[my circ] at (\i*2.414,\j*2.414) {};
    \node(h)[my squa] at (\i*2.414+1,\j*2.414) {};
    }
    \foreach \i in {0,...,4} 
    \foreach \j in {0,...,2} {
    \node(h)[my tri] at (\i*2.414-0.71,\j*2.414+0.71) {};
    \node(h)[my tri2] at (\i*2.414-0.71,\j*2.414+1.71) {};
    }
\end{tikzpicture}&
\begin{tikzpicture}
  \foreach \n[count=\k] in {4,5,4}{
    \foreach \m in {1,...,\n}{
      \node(h)[my hex] at (-\n/2+\m,{\k*sqrt(3)/2}){};
      \foreach \t in {1,4} \node[my circ] at ($(h)+({(\t-1)*60}:.5)$){};
      \foreach \t in {2,5} \node[my tri] at ($(h)+({(\t-1)*60}:.5)$){};
      \foreach \t in {3,6} \node[my squa] at ($(h)+({(\t-1)*60}:.5)$){};
    }
  }
\end{tikzpicture}\\
The octogonal lattice&The Kagome lattice
\end{tabular}
\caption{\label{kag} The octogonal lattice (4 points) and the Kagome lattice (3 points)
are periodic graphs.}

\end{figure}
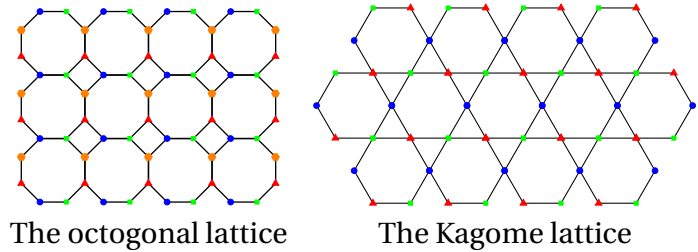
\end{center}




Note that the periodicity assumption implies that a 1-point periodic graph is homogeneous and moreover that all vertices have even degree $2k$ for some $k\geq d$ since the graph is connected.
    Therefore $h\Gamma$ is a $2k$-degree graph and we can write
\begin{equation*}
    h\mathcal{E}=\left\{\left(x,x\pm h\mathbf{e}_j\right) \ | \ x\in h\mathcal{V},\ j=1,\ldots,k\right\},
\end{equation*}
where the vectors $e_j$ are of the form
$e_j=\mathbf{v}(n_j)$ for some $n_j\in\Z^d$. Note that the system
$(e_1,\ldots,e_k)$ is spanning $\R^d$, otherwise the graph $h\Gamma$ would not be 
connected.

Furthermore, writing the vectors $\mathbf{e}_j$
as columns, we define the
$k\times d$ matrix
\begin{equation*}
    E=\begin{pmatrix}
    e_1 & e_2 & \dots & e_k
    \end{pmatrix}. 
\end{equation*}

\begin{lemma}
The Gramm matrix of the system $(e_1,\ldots,e_k)$, $G=E^*E$ is positive definite, thus 
\begin{equation*}
    |x|_\Gamma:=\sqrt{x^*\left(E^*E\right)^{-1}x}
\end{equation*}
defines a norm on $\R^d$.
\end{lemma}

\begin{proof} This is a classical fact about Gramm matrices, let us prove it for the sake of 
completeness. First $G$ is symmetric. Further $\ker (E)\subset\ker (G)$ while 
\begin{equation*}
    Gx=0 \implies x^*E^*Ex=0 \implies (Ex)^*Ex=0 \implies Ex=0,
\end{equation*}
so that $\ker(G)=\ker(E)$. By the rank theorem we get
\begin{equation*}
    d=\mbox{rank}(E)+\dim\ker(E)=d+\dim\ker(E)
\end{equation*}
since the columns of $E$ span $\R^d$. It follows that  $\ker(G)=\ker(E)=\{0\}$
and $G$ is positive definite thus invertible and its inverse is positive definite as well.
\end{proof}

The associated balls will be denoted as
\begin{equation}\label{ball:1}
    B_r = \left\{ x \in \R^d \,: \  |x|_{\Gamma}\leq r \right\}.
\end{equation}

Now let $f:h\mathcal{V}\mapsto\C$ be a function defined on a graph $h\Gamma$. Let $\ell^2(h\mathcal{V})$ be the set of functions $f:h\mathcal{V}\mapsto\C$ that satisfy
\begin{equation*}
    \norm{f}_{\ell^2(h\mathcal{V})}^2:=\sum_{v\in h\mathcal{V}}\abs{f(v)}^2<\infty,
\end{equation*}
and for $f,g\in \ell^2(h\mathcal{V})$ we define the inner product,
\begin{equation*}
    \scal{f,g}_{\ell^2(h\mathcal{V})}:=\sum_{v\in h\mathcal{V}}f(v)\overline{g(v)}.
\end{equation*}
We may simply write $\norm{f}=\norm{f}_{\ell^2(h\mathcal{V})}$ and 
$\scal{f,g}=\scal{f,g}_{\ell^2(h\mathcal{V})}$ when the graph and the mesh size $h$ are fixed.

It remains to define several operators on $h\Gamma$, the first being the graph Laplacian. 
\begin{definition}
Let $f$ be a function on a graph $h\Gamma$, then
\begin{equation*}
    \Delta_{h\Gamma} f(x)=\sum_{x\sim y}\left(f(y)-f(x)\right), \quad x\in h \mathcal{V}.
\end{equation*}
We say that $f$ is a discrete harmonic function (or simply is harmonic) if 
$$
\Delta_{h\Gamma}f(x)=0
$$ for all $x\in h\mathcal{V}$.
\end{definition}

Assume $h\Gamma$ is a 1-point periodic graph, then we define
\begin{equation*}
    D_j^\pm f(x):= \pm\left(f(x\pm he_j)-f(x)\right)\quad\mbox{and}\quad D^s_j:=f(x+ he_j)-f(x- he_j),
\end{equation*}
and $\Delta_j=D^+_jD^-_j=D^-_jD^+_j$.
We see that the Laplacian of a 1-point periodic graph can be written as
\begin{eqnarray*}
    \Delta_{h\Gamma}f(x)= \sum_{j=1}^k D^+_jD_j^- f(x)&=&\sum_{j=1}^k\left(f(x+ he_j)+f(x- he_j)-2f(x)\right) \\
    &=:& \sum_{j=1}^k \Delta_j f(x).
\end{eqnarray*}
Finally, we define
$$
D_s u(x)=\left(\sum_{j=1}^k |D_j^s u(x)|^2\right)^{\frac{1}{2}}
\quad\mbox{and}\quad 
D_s^2u=\left(\sum_{j=1}^k |(D_j^s)^2 u(x)|^2\right)^{\frac{1}{2}}.
$$

Further, consider the following vectors and matrices,
\begin{equation*}
    x_E=Ex, \quad \nabla^2_E f=E\left(\nabla^2f\right)E^*,
\end{equation*}
and we denote $\abs{x}^2_E=x^*E^*Ex$.
Finally, let $\mathbb{T}_h^d=\R^d/(h^{-1}\Z)^d$ and define the operator $\mathcal{U}:\ell^2(h\mathcal{V})\mapsto \ell^2(\mathbb{T}_h^d)$,
\begin{equation*}
 \mathcal{U}f(\xi) := \hat{f}(\xi)=\sum_{v\in h\mathcal{V}}f(v)e^{2i\pi \scal{v,\xi}}.
\end{equation*}
It is worth noting that it satisfies a Parseval type equality, 
\begin{equation*}
    \norm{\hat{f}}_{\ell^2(\mathbb{T}_h^d)}^2:=h^d\int_{[0,h^{-1}]^d}|\hat{f}(\xi)|^2\,\mbox{d}\xi=\norm{f}_{\ell^2(h\mathcal{V})}^2.
\end{equation*}
Further
$$
\mathcal{U}\Delta_jf(\xi)=2\bigl(\cos \left(2\pi h\scal{e_j,\xi}\right)-1\bigr)\hat{f}(\xi)
=-4\sin^2\left(\pi h\scal{e_j,\xi}\right)\hat{f}(\xi)
$$
and
$$
\mathcal{U}D_j^sf(\xi)=-2i\sin 2\pi h\scal{e_j,\xi}.
$$
\section{A Carleman Estimate}

Our first goal in this paper is to prove the following Carleman estimate for 1-point periodic graphs. 
\begin{theorem}[Carleman Estimate]\label{Carleman:1}
Let $h\Gamma$ be a 1-point periodic graph.
    Let $u:(h\Gamma)\to \R$ with $\text{supp}(u) \subset B_2 \setminus B_{1/2}$. Let $g$ be defined by $g=h^{-2}\Delta_{h\Gamma} u$ in $B_4$ and $g=0$ on $h\mathcal{V}\setminus B_4$.
    Let
    \begin{equation*}
        \varphi(t)=-\log t + c \left( \log t \arctan(\log t) -\frac{1}{2} \log (1+ \log^2 t )\right)
    \end{equation*}
    for a certain small constant $c>0$ and $\phi_\tau(x):=\tau\varphi(|x|_\Gamma))$
    
    Then, there exist $h_0,\delta_0 \in (0,1)$ with $h_0 < \delta_0$, $C>1$ and $1<\tau_0<\delta_0h_0^{-1}$ such that for all $h\in (0,h_0)$ and $\tau \in (\tau_0, \delta_0h^{-1})$ we have
    \begin{multline}\label{carleman:eq1}
        \tau^3 \norm{e^{ \phi_\tau} u}^2_{\ell^2(h\mathcal{V})} + \tau \norm{e^{\phi_\tau} h^{-1}D_s u}^2_{\ell^2(h\mathcal{V})} + \tau^{-1} \norm{e^{ \phi_\tau} h^{-2}D^2_s u}^2_{\ell^2(h\mathcal{V})} \\
        \leq C \norm{e^{\phi_\tau} g}^2_{\ell^2(h\mathcal{V})}.
    \end{multline}
\end{theorem}

To prove the Three Balls Inequality in the next section, we will further lower bound the 
left hand side and only use the less precise version
    \begin{equation}\label{carleman:eq1bis}
\tau^3\norm{e^{ \phi_\tau} u}^2_{\ell^2(h\mathcal{V})}   
+ \tau \norm{e^{\phi_\tau} h^{-1}D_s u}^2_{\ell^2(h\mathcal{V})} 
        \leq C \norm{e^{\phi_\tau} g}^2_{\ell^2(h\mathcal{V})}.
    \end{equation}

The remaining of this section is devoted to the proof of this theorem. Before we start proving the Carleman Estimate, we remark that we can rescale the graph $\Gamma$ by a constant. In particular, it will be useful to scale the graph $\Gamma$ in such a way that 
\begin{equation}\label{scaling:1}
    \norm{e_j}_2<\frac{1}{4\sqrt{d}},
\end{equation}
for all $j=1,\dots,k$. We are now ready to prove Theorem \ref{Carleman:1}.
First we will introduce a lemma on the Carleman weight $\phi$, and then introduce a new operator $L_\phi$. 

\begin{lemma}\label{lemma:3.1}
Let $0<c<\dfrac{2}{\pi}$,
 \begin{equation*}
        \varphi(t)= -\log(t)+c\left(\left(\log t\right)\arctan\left(\log t\right)- \frac{1}{2}\log\left(1+\left(\log t\right)^2\right)\right),
    \end{equation*}
and $\phi_\tau(x):=\tau\varphi\left(|x|_\Gamma\right)$. Further, let $\gamma_0>0$ and define    
    $$
    \mathcal{C}_\tau=\left\{\abs{\xi_E}^2=\abs{\nabla_E\phi_\tau(x)}^2\right\}\cap\left\{\scal{\xi_E,\nabla_E\phi_\tau(x)}_{\R^k}=0\right\},
    $$
as well as
    \begin{equation*}
    \mathcal{N}_{\tau,\mathcal{C}}:=\left\{\xi :\mbox{dist}(\xi_E,\mathcal{C}_\tau)\leq\gamma_0\tau\right\}.
\end{equation*}
Then, taking $\gamma_0$ small enough, there exist a constant $C>0$ such that if $x\in B_4\setminus B_1$ and $\xi\in \mathcal{N}_{\tau,\mathcal{C}}$, we have
        \begin{equation}\label{lemma eq:3.1}
        \nabla_E^*\phi_\tau(x)\nabla^2_E \phi_\tau(x)\nabla_E\phi_\tau(x)+ \xi_E^*\nabla^2_E\phi_\tau(x)\xi_E \geq C\tau^3.
    \end{equation} 
\end{lemma}

\begin{proof}
As \eqref{lemma eq:3.1} is homogeneous of degree $3$ with respect to $\tau$, it is enough to prove it for $\tau=1$.
We will simplify the notation by writing $\phi=\phi_1$. First we assume that $\xi\in\mathcal{C}_\tau$.

Define $A:=\left(E^TE\right)^{-1}$, we have
$$
\nabla\phi(x)=\dfrac{\varphi'\left(|x|_\Gamma\right)}{|x|_\Gamma}Ax
$$
thus, after multiplying by $E$ we get
\begin{equation}
\label{eq:phi*}
    \nabla_E\phi(x)=\frac{\varphi'\left(|x|_\Gamma\right)}{|x|_\Gamma}EAx.
\end{equation}
Further
$$
\nabla^2\phi(x)=\left(\varphi''\left(|x|_\Gamma\right)-\dfrac{\varphi'\left(|x|_\Gamma\right)}{|x|_\Gamma}\right)\left(\frac{Ax}{|x|_\Gamma}\otimes \frac{Ax}{|x|_\Gamma}\right)+\dfrac{\varphi'\left(|x|_\Gamma\right)}{|x|_\Gamma}A.
$$
After multiplying on the left by $E$ and on the right by $E^*$, we get
\begin{equation*}
    \nabla^2_E\phi(x)=\left(\varphi''\left(|x|_\Gamma\right)-\frac{\varphi'\left(|x|_\Gamma\right)}{|x|_\Gamma}\right)E\left(\frac{Ax}{|x|_\Gamma}\otimes\frac{Ax}{|x|_\Gamma}\right)E^*
    +\frac{\varphi'\left(|x|_\Gamma\right)}{|x|_\Gamma}EAE^*.
\end{equation*}
And hence, the LHS of \eqref{lemma eq:3.1} becomes
\begin{align*}
\left(\varphi''\left(|x|_\Gamma\right)-\frac{\varphi'\left(|x|_\Gamma\right)}{|x|_\Gamma}\right)
\left(\varphi'(|x|_\Gamma)^2+\scal{x,\xi}_{\R^d}^2\right)
+\frac{\varphi'\left(|x|_\Gamma\right)}{|x|_\Gamma}\left(\varphi'\left(|x|_\Gamma\right)^2+\xi^*E^*E\xi\right).
\end{align*}

Next notice that from \eqref{eq:phi*}, $\abs{\xi_E}^2=\abs{\nabla_E\phi(x)}^2$ if and only if
\begin{equation*}
\xi^*E^*E\xi=\left(\frac{\varphi'\left(|x|_\Gamma\right)}{|x|_\Gamma}\right)^2x^*A^*E^*EAx=\varphi'\left(|x|_\Gamma\right)^2,
\end{equation*}
while $\scal{\xi_E,\nabla_E\phi(x)}_{\R^k}=0$ reduces to $\scal{x,\xi}_{\R^d}=0$.
Therefore, we end up with
\begin{align*}
    \nabla_E^*\phi\nabla^2_E\phi\nabla_E\phi+ \xi_E^*\nabla^2_E\phi\xi_E &= \left(\varphi''\left(|x|_\Gamma\right)+\frac{\varphi'\left(|x|_\Gamma\right)}{|x|_\Gamma}\right)
\varphi'(|x|_\Gamma)^2 \\
    &=c\frac{\left(1-c\arctan\log|x|_\Gamma\right)^2}{|x|_\Gamma^4\left(1+\log^2|x|_\Gamma\right)}\\
    &\geq C:=c\frac{\left(1-c\pi/2\right)^2}{4^4\left(1+4\log^22\right)}
\end{align*}
where $0<c<\dfrac{2}{\pi}$ when $1\leq|x|_\Gamma \leq 4$.

We next seek to argue that by continuity a similar lower bound also holds on $\mathcal{N}_{\tau,\mathcal{C}}$. To this end, note that for $\xi\in\mathcal{N}_{\tau,\mathcal{C}}$ we have $\abs{\xi}_E\in(C_{0,1},C_{0,2})$, where the constants only depend on $\gamma_0$ and on the graph $\Gamma$ and, in particular, 
they are independent of $\tau$ and $h$. Since for $\xi\in\mathcal{C}_\tau$, the condition on $\phi$ gives 
$$
\nabla_E^*\phi\nabla^2_E\phi\nabla_E\phi+ \xi_E^*\nabla^2_E\phi\xi_E \geq C >0,
$$ 
by continuity, on a sufficiently small neighborhood $\mathcal{N}_{\tau,\mathcal{C}}$, {\it i.e.} if $\gamma_0>0$
is small enough, we have 
$$
\nabla_E^*\phi\nabla^2_E\phi\nabla_E\phi+ \xi_E^*\nabla^2_E\phi\xi_E \geq \frac{C}{2} >0,
$$
which concludes the proof.
\end{proof}

Now let us define the operator $L_{\phi_\tau} : \ell^2(h\mathcal{V}) \mapsto \ell^2(h\mathcal{V})$.
\begin{multline}
 \label{eq:deflphitau}
    L_{\phi_\tau} f(x):=
    h^{-2}e^{\phi_\tau}\Delta_{h\Gamma} e^{-\phi_\tau}f(x)\\
    =  h^{-2} \sum_{j=1}^k\bigg(\cosh\left(D^+_j \phi_\tau(x)\right)f(x+he_j)-\sinh\left( D^+_j \phi_\tau(x)\right)f(x+he_j)\\
    +\cosh\left( D^-_j \phi_\tau(x)\right)f(x-he_j)+\sinh \left(D^-_j \phi_\tau(x)\right)f(x-he_j)-2f(x)\bigg).
\end{multline}
Its adjoint is obtained as follows:
\begin{align*}
    &\scal{L_{\phi_\tau} f,g}\\
    =& \frac{1}{h^2}\sum_{x\in h\mathcal{V}}\sum_{j=1}^k\bigg(\cosh\left(D^+_j \phi_\tau(x)\right)f(x+he_j)-\sinh\left( D^+_j\phi_\tau(x)\right)f(x+he_j) \\
    &+\cosh\left(D^-_j \phi_\tau(x)\right)f(x-he_j)+\sinh \left( D^-_j \phi_\tau(x)\right)f(x-he_j)-2f(x)\bigg)\overline{g(x)} \\
    =&\frac{1}{h^2}\sum_{x\in h\mathcal{V}}\sum_{j=1}^k f(x)\bigg(\cosh\left( D^-_j \phi_\tau(x)\right)\overline{g(x-he_j)}-\sinh\left( D^-_j \phi_\tau(x)\right)\overline{g(x-he_j)} \\
    &+\cosh\left(D^+_j \phi_\tau(x)\right)\overline{g(x+he_j)}+\sinh \left( D^+_j \phi_\tau(x)\right)\overline{g(x+he_j)}-2\overline{g(x)}\bigg)\\
    :=&\scal{f,L_{\phi_\tau}^*g}.
\end{align*}
Here we used that the graph is 1-point periodic so that the vertex set is reduced to $h\mathcal{L}$
and thus, if $x\in h\mathcal{V}$, so are $x+he_j$ and $x-he_j$.
We now decompose $L_{\phi_\tau}$ as a sum of a symmetric and an antisymmetric operator
$L_{\phi_\tau}=S_{\phi_\tau}+A_{\phi_\tau}$ with $S_{\phi_\tau}=\dfrac{L_{\phi_\tau}+L_{\phi_\tau}^*}{2}$ and $A_{\phi_\tau}=\dfrac{L_{\phi_\tau}-L_{\phi_\tau}^*}{2}$.
Explicitely,
\begin{equation*}
   S_{\phi_\tau}= \sum_{j=1}^kS_j,\quad A_{\phi_\tau}=\sum_{j=1}^kA_j,
\end{equation*}
with 
\begin{align}
    S_j f(x) =& \frac{1}{h^{2}} \Bigl(\cosh(D^+_j\phi_\tau(x))f(x+he_j)\notag\\
    &\qquad+\cosh(D^-_j\phi_\tau(x))f(x-he_j)-2f(x) \Bigr),\label{operatorsj1}\\
    =&\frac{1}{h^{2}}\Bigl(\bigl(\cosh(D^+_j\phi_\tau(x))-1\bigr)f(x+he_j)\notag\\
    &\qquad+\bigl(\cosh(D^-_j\phi_\tau(x))-1\bigr)f(x-he_j)+\Delta_jf(x)\Bigr),\label{operatorsj2}
  \end{align}
 while $A_jf(x)$ is given by
 \begin{equation}
    h^{-2} \left(\sinh(D^-_j\phi_\tau(x))f(x-he_j)-\sinh(D^+_j\phi_\tau(x))f(x+he_j) \right).\label{operatoraj} 
 \end{equation}
    
Let us now introduce a small parameter $\lambda>0$ to be fixed later (not too small).
We write $\mathcal{B}$ for the smallest ball containing all $e_j$'s.

\begin{definition}
    For $1\leq \tau\leq \dfrac{1}{h}$, a $\tau$-partition of unity for $B_2\setminus B_{1/2}$ will be a
    family of functions $\{\psi_n\}$ such that:

    -- there is a covering of $B_2\setminus B_{1/2}$ by a family $\{B_n\}$
of balls of radius $\tau^{-\frac{1}{2}}\lambda$ and we assume that the family $\{B_n^h:=B_n+h\mathcal{B}\}$
has a finite covering.

    -- $\{\psi_n\}$ is a smooth partition of unity subordinated to these balls, {\it i.e.} 
\begin{enumerate}
\renewcommand{\theenumi}{\roman{enumi}}
    \item $\psi_n\geq0$ and $\sum\psi_n=1$ on $B_2\setminus B_{1/2}$;

    \item $\supp\psi_n\subset B_n$;

    \item $\abs{\nabla\psi_n(x)}\lesssim \tau^{\frac{1}{2}} \lambda^{-1} \chi_{\supp_{B_n^h}(\psi_n)}(x)$;
    
    \item $\abs{\nabla^2\psi_n(x)}\lesssim \tau^{1} \lambda^{-2} \chi_{\supp(\psi_n)}(x)$.
\end{enumerate}
\end{definition}

\begin{lemma}\label{lemma:3.2}    There are $h_0,\delta_0\in(0,1)$ small enough such that, for $h\in(0,h_0)$, 
$\tau\in(1,\delta_0h^{-1})$, if
\begin{enumerate}
\renewcommand{\theenumi}{\roman{enumi}}
    \item $(\psi_n(x))_{n\in\Z}$ is a $\tau$-partition of unity for $B_2 \text{\textbackslash} B_{1/2}$;

    \item $\phi_\tau$ is given in Lemma \ref{lemma:3.1} and $L_{\phi_\tau}$ given by \eqref{eq:deflphitau};
\end{enumerate}
then for every $f\in \ell^2(h\mathcal{V})$, writing $f_n=f\psi_n$, we have
    \begin{multline*}
        \norm{L_\phi f}\leq \sum_n \norm{L_\phi f_n}\\ \leq C\norm{L_\phi f}+C\tau^{\frac{1}{2}}\lambda^{-1}\sum_{j=1}^k\norm{h^{-1}D_j^sf}+C\left(\tau\lambda^{-2}+\tau^{\frac{3}{2}}\lambda^{-1}+\tau^{\frac{5}{2}}h\lambda^{-1}\right)\norm{f}
    \end{multline*}
    where $c$ is a constant depending only on $\psi_n$ and $\varphi$.
\end{lemma}

\begin{proof} The first inequality follows directly from $f=\sum f_n$.

We will use the expression \eqref{operatorsj2} for $S_{\phi_\tau}$. A simple computation shows that
$\Delta_jf_n(x) =$
\begin{align*}
        \psi_n(x)\Delta_jf(x)+D^s_jf(x) \left(\psi_n(x+he_j)-\psi_n(x)\right)+f(x-he_j)\Delta_j\psi_n(x).
    \end{align*}
The second term is bounded by
$$
    \begin{aligned}
        \bigl|\bigl(f(x+he_j)-f(x-he_j)\bigr)& \bigl(\psi_n(x+he_j)-\psi_n(x)\big)\bigr| \\ 
        &\lesssim h \abs{\left(f(x+he_j)-f(x-he_j)\right)}\abs{\nabla \psi_n} \\
        &\lesssim h\tau^{\frac{1}{2}}\lambda^{-1} \abs{D^s_jf(x)}\chi_{\supp(\psi_n)}(x)
    \end{aligned}
$$
while the last one is bounded by
    \begin{eqnarray*}
        \abs{f(x-he_j)\Delta_j\psi_n(x) }&\lesssim& h^2\abs{f(x-he_j)}\abs{\nabla^2\psi_n}\\
        &\lesssim& \tau \lambda^{-2}\abs{f(x-he_j)}\chi_{\supp(\psi_n)}(x).
    \end{eqnarray*}
    We infer that
    \begin{multline}
        h^{-2}\abs{\Delta_j f_n(x)-\psi_n(x)\Delta_jf(x)}\\
        \lesssim
        h^{-1}\tau^{\frac{1}{2}}\lambda^{-1} \abs{D^s_jf(x)}\chi_{\supp(\psi_n)}(x)
        +\tau \lambda^{-2}\abs{f(x-he_j)}\chi_{\supp(\psi_n)}(x).
        \label{eq:sphitau1}
    \end{multline}  
   
    Similarly for $h^{-2}\left(\cosh(D^\pm_j\phi_\tau(x))-1\right)f_n(x\pm he_j)$ we obtain
    \begin{align*}
        h^{-2}&\bigl(\cosh(D^\pm_j\phi_\tau(x))-1\bigr)f_n(x\pm he_j)\\
        =&h^{-2}\left(\cosh(D^\pm_j\phi_\tau(x))-1\right)f(x\pm he_j)\psi_n(x\pm he_j) \\
        =& \psi_n(x)h^{-2}\left(\cosh(D^\pm_j\phi_\tau(x))-1\right)f(x\pm he_j)\\
        &+h^{-2}\left(\cosh(D^\pm_j\phi_\tau(x))-1\right)f(x\pm he_j)\left(\psi_n(x\pm he_j)-\psi_n(x)\right).
    \end{align*}
    The second term is the error term that we now estimate.
    
    First, there is a $y\in [x,x\pm he_j]$ such that 
    \begin{equation}\label{eq:lemma:3.3}
        \abs{D^\pm_j \phi_\tau(x)}=h\abs{\scal{\nabla\phi_\tau(y),e_j}_{\R^d}}=h\tau\abs{\scal{\nabla\varphi(|y|_\Gamma),e_j}_{\R^d}}\lesssim h\tau.
    \end{equation}  
    As $h\tau\leq 1$ and $\cosh t-1\lesssim t^2$ on $[0,C]$,
    \begin{align*}
        \abs{\cosh(D^\pm_j\phi_\tau(x))-1}\lesssim \abs{D^\pm_j\phi_\tau(y)}^2\lesssim \tau^2h^2,
    \end{align*}
    where $y$ is treated as an intermediate value. Hence
    \begin{align*}
        \big|h^{-2}\left(\cosh(D^\pm_j\phi_\tau(x))-1\right)&f(x\pm he_j)\left(\psi_n(x\pm he_j)-\psi_n(x)\right)\big| \\
        &\lesssim \tau^{\frac{5}{2}}\lambda^{-1} h\abs{f(x\pm he_j)}\chi_{ \supp(\psi_n)}(x).
    \end{align*}

    Finally, together with \eqref{eq:sphitau1}, we obtain that
    \begin{align*}
    |S_{\phi_\tau}f_n-\psi_nS_{\phi_\tau}f|\leq& 
    Ch^{-1}\tau^{\frac{1}{2}}\lambda^{-1} \abs{D^s_jf(x)}\chi_{\supp(\psi_n)}(x)\\
    &+C\tau \lambda^{-2}\abs{f(x-he_j)}\chi_{\supp(\psi_n)}(x)\\
    &+2C\tau^{\frac{5}{2}}\lambda^{-1} h\abs{f(x\pm he_j)}\chi_{ \supp(\psi_n)}(x).  
    \end{align*}
    
    Now for $A_{\phi_\tau}$, we can write
    \begin{align*}
        A_{\phi_\tau}&f_n(x)\\
        =&\psi_n(x) A_{\phi_\tau}f(x)+h^{-2}\sinh(D^-_j\phi_\tau(x))f_n(x-he_j)\left(\psi_n(x- he_j)-\psi_n(x)\right)\\
        &-h^{-2}\sinh(D^+_j\phi_\tau(x))f(x+he_j)\left(\psi_n(x+ he_j)-\psi_n(x)\right).
    \end{align*}
    We now use that $|\sinh t|\lesssim t$ on $[0,C]$ to get
    \begin{align*}
        \abs{\sinh(D^\pm_j\phi_\tau(x))}\lesssim \tau h,
    \end{align*}
    so that
    \begin{multline*}
         \abs{h^{-2}\sinh(D^\pm_j\phi_\tau(x))f_n(x\pm he_j)\left(\psi_n(x\pm he_j)-\psi_n(x)\right)}\\
         \lesssim \tau h^{-1}\abs{f(x\pm he_j}\abs{\nabla\psi_n} \lesssim \tau^{\frac{3}{2}}\lambda^{-1}\abs{f(x-e_j)}\chi_{\supp(\psi_n)}(x). 
    \end{multline*}  

    Combining everything into the expression on $L_{\phi_\tau}$ we obtain that
    
    \begin{align*}
        \sum_n \norm{L_{\phi_\tau} f_n} \leq& C \sum_n\norm{\psi_n L_{\phi_\tau} f} + \sum_n \left(C\tau^{\frac{1}{2}}\lambda^{-1}\sum_{j=1}^k\norm{h^{-1}D_j^sf \chi_{ \supp(\psi_n)}}\right.\\
        &\left.\phantom{\sum_{j=1}^k}+C(\tau\lambda^{-2}+\tau^{\frac{3}{2}}\lambda^{-1}+\tau^{\frac{5}{2}}h\lambda^{-1})\norm{f\chi_{ \supp(\psi_n)}} \right) \\
    \leq& C\norm{L_{\phi_\tau} f} + C\tau^{\frac{1}{2}}\lambda^{-1}\sum_{j=1}^k\norm{h^{-1}D_j^sf}\\
    &+C\left(\tau\lambda^{-2}+\tau^{\frac{3}{2}}\lambda^{-1}+\tau^{\frac{5}{2}}h\lambda^{-1}\right)\norm{f}
    \end{align*}
as claimed.    
\end{proof}

We then compute the commutator which writes
\begin{multline}\label{commutator:eq:1}
    [S_{\phi_\tau},A_{\phi_\tau}]f(x)=\sum_{i=1}^k\sum_{j=1}^k\sum_{\alpha,\beta=\pm 1} \\
    -\beta h^{-4}\bigr[\cosh(D_j^\alpha\phi_\tau(x))\sinh(D_i^\beta \phi_\tau(x+h\alpha e_j)) \\
    -\cosh(D_j^\alpha\phi_\tau(x+h\beta e_i))\sinh(D_i^\beta \phi_\tau(x))\bigl]f(x+h\alpha e_j+h\beta e_i).
\end{multline}
Since we have
$$
D_j^\alpha \phi_\tau(x+h\beta e_i)=\beta D_j^\alpha D_i^\beta \phi_\tau(x)+D_j^\alpha \phi_\tau(x),
$$
using hyperbolic trigonometry identities, the bracket in \ref{commutator:eq:1} writes
\begin{align*}
\cosh\left(D_j^\alpha\phi_\tau(x)\right)&\sinh\left(D_i^\beta \phi_\tau(x+h\alpha e_j)\right)\\
-\cosh&\left(D_j^\alpha\phi_\tau(x+h\beta e_i)\right)\sinh\left(D_i^\beta \phi_\tau(x)\right) \\
    =& \cosh\left(D_j^\alpha\phi_\tau(x)\right)\sinh\left(\alpha D_j^\alpha D_i^\beta \phi_\tau(x)\right)\cosh\left(D_i^\beta \phi_\tau(x)\right) \\
    &+ \cosh\left(D_j^\alpha\phi_\tau(x)\right)\cosh\left( D_j^\alpha D_i^\beta \phi_\tau(x)\right)\sinh\left(D_i^\beta \phi_\tau(x)\right) \\
    &- \sinh\left(D_i^\beta\phi_\tau(x)\right)\cosh\left( D_j^\alpha D_i^\beta \phi_\tau(x)\right)\cosh\left(D_j^\alpha \phi_\tau(x)\right) \\
    &- \sinh\left(D_i^\beta\phi_\tau(x)\right)\sinh\left(\beta D_j^\alpha D_i^\beta \phi_\tau(x)\right)\sinh\left(D_i^\beta \phi_\tau(x)\right) \\
    =& \sinh\left(\alpha D_j^\alpha D_i^\beta \phi_\tau(x)\right)\cosh\left(D_j^\alpha\phi_\tau(x)\right)\cosh\left(D_i^\beta \phi_\tau(x)\right) \\
    &- \sinh\left(\beta D_j^\alpha D_i^\beta \phi_\tau(x)\right)\sinh\left(D_i^\beta\phi_\tau(x)\right)\sinh\left(D_j^\alpha \phi_\tau(x)\right) \\
    =& \alpha \sinh\left( D_j^\alpha D_i^\beta \phi_\tau(x)\right)\cosh\left( \alpha D_j^\alpha \phi_\tau(x) - \beta D_i^\beta \phi_\tau(x)\right).
\end{align*}

Considering the product $ \langle [S_{\phi_\tau},A_{\phi_\tau}]f,f\rangle$, and translating $x$ to $x-\beta h e_i$ and also changing $\beta$ to $-\beta$ we obtain
\begin{multline}
    \langle [S_{\phi_\tau},A_{\phi_\tau}]f,f\rangle=   
    \frac{1}{h^4}\sum_{x\in h\mathcal{V}}\sum_{i=1}^k\sum_{j=1}^k\\\sum_{\alpha,\beta=\pm1} 
     \alpha\beta\sinh\left(D^\alpha_iD^\beta_j\phi_\tau(x)\right)\cosh\left(D_{i,j}^{\alpha,\beta}\phi_\tau(x)\right)f(x+\alpha he_i)\overline{f(x+\beta he_j)}    
\end{multline}
where $D_{i,j}^{\alpha,\beta}\phi_\tau(x)=\phi_\tau(x+\alpha he_i+\beta he_j)-\phi_\tau(x)$.

For the next lemma, we will approximate the hyperbolic identities using Taylor expansion 
which allows for easier manipulations. 

\begin{lemma}\label{lemma:3.3}
    Let $\phi_\tau$ be as in lemma \ref{lemma:3.1}. Let $S_j, A_j$ and $[S_i,A_j]f(x)$ be the quantities defined beforehand. Let
    \begin{align*}
        \Tilde{S}_jf(x) :=& h^{-2}\Delta_jf(x) + \frac{\scal{\nabla\phi_\tau(x),e_j}^2}{2}\left(f(x+he_j)+f(x-he_j)\right),\\
        \Tilde{A}_jf(x) :=& -h^{-1}\scal{\nabla\phi_\tau(x),e_j}D_j^sf(x), \\
        \mathbf{C}_{i,j}^{f,f}(x) :=& h^{-2}e_i^*\nabla^2\phi_\tau(x)e_jD_i^sf(x)\overline{D_j^sf(x)}\\
        &+ \frac{1}{2}\sum_{\alpha,\beta=\pm} \alpha \beta e_i^*\nabla^2\phi_\tau(x)e_j\abs{\scal{\nabla\phi_\tau(x),\alpha e_i+\beta e_j}}^2\\
        &\qquad\qquad\times f(x+\alpha he_i)\overline{f(x+\beta he_j)}.
    \end{align*}
    Let $\tau \in (1,h^{-1}\delta_0)$, where $h\in(0,h_0)$ with $\delta_0\in(0,1)$ and $h_0<\delta_0$. Then, for $\Tilde{S}_{\phi_\tau} f(x):=\dst\sum_{j=1}^k \Tilde{S}_j f(x)$ and $\Tilde{A}_{\phi_\tau} f(x):=\dst\sum_{j=1}^k \Tilde{A}_j f(x)$, and a function $f$ supported in $B_2 \setminus B_{1/2}$, we have
    \begin{align*}
        \norm{S_{\phi_\tau} f- \Tilde{S}_{\phi_\tau} f}&\lesssim(\delta_0\tau+\delta_0^2\tau^2)\norm{f},\\
        \norm{A_{\phi_\tau} f-\Tilde{A}_{\phi_\tau} f} &\lesssim(\tau+\delta_0\tau^2)\norm{f},\\
        \scal{[S_{\phi_\tau},A_{\phi_\tau}]f,f}- \sum_{x\in(h\mathcal{V})} \sum_{i,j=1}^k\mathbf{C}_{i,j}^{f,f}(x)&\lesssim (\tau^2+\delta_0^2\tau^3)\norm{f}^2 +\sum_{j=1}^k\norm{h^{-1}D_j^sf}^2.
    \end{align*}
\end{lemma}

\begin{proof}
We will use that for every $A$ there is a $B$ such that, 
\begin{equation}
    \label{eq:trivestcosh}
    \abs{\cosh x-1-\dfrac{x^2}{2}}\leq B|x|^4\qquad\mbox{when }|x|\leq A
\end{equation}.

    First we deal with the symmetric operator $S_{\phi_\tau}$. To do so, obseve first that, with the mean value property,
    for every $x$ and $j$, there is a $\xi$ in the segment $[x,x\pm he_j]$ such that
    $$
    \cosh(D_j^\pm \phi_\tau(x))-1=\frac{h^2}{2}\scal{\nabla\phi_\tau(\xi),e_j}^2+E_1
    $$
    with $|E_1|\leq Ch^4\tau^4$ with $C$ depending on $\ffi'$ only. Here we use that $x\pm he_i$ stays in $B_4\setminus B_{1/4}$ thus so is $\xi$ and therefore $\scal{\nabla\phi_\tau(\xi),e_j}^2$ is bounded by some fixed constant $A$
    independent of $x,j$.
    
    Further, using the mean value property again, 
    $$
    \scal{\nabla\phi_\tau(\xi),e_j}^2=\scal{\nabla\phi_\tau(x),e_j}^2+E_2
    $$
    with $|E_2|\leq Ch\tau^2$ for some constant $C$ that depends on $\ffi$ and $\ffi'$ only.
    It follows that
$$
\cosh(D_j^\pm \phi_\tau(x))-1=\frac{h^2}{2}\scal{\nabla\phi_\tau(x),e_j}^2+E_3
$$
with $|E_3|\leq C(h^3\tau^2+h^4\tau^4)$. Summing over $x$ and $j$, we obtain
    \begin{equation*}
        \|S_{\phi_\tau}f-\Tilde{S}_{\phi_\tau}f\|\lesssim \frac{h^3\tau^2+h^4\tau^4}{h^2}\|f\|
        \leq C(\delta_0\tau+\delta_0^2\tau^2)\|f\|
    \end{equation*}
    since $h\tau\leq\delta_0$.

\smallskip

The anti-symmetric operator $A_{\phi_\tau}$ is dealt with in a similar way but using that,
for every $A>0$ there is a $B$ such that $|\sinh x-x|\leq B|x|^3$ when $|x|\leq A$. We then get that
$$
\sinh(D_j^\alpha\phi_\tau(x))=\alpha h\scal{\nabla\phi_\tau(x),e_j}+E_4
$$
with $|E_4|\leq C(h^2\tau+h^3\tau^3)$. Summing over $x$ and $j$, we obtain the desired estimate.

\smallskip

    Finally, we consider the commutator. First, using again \eqref{eq:trivestcosh}
    and approximating the difference operator $D^{\alpha,\beta}_{i,j}\phi_\tau(x)$ as before we get
    $$
    \cosh\left(D_{i,j}^{\alpha,\beta} \phi_\tau(x)\right)-1=\frac{h^2}{2}\scal{\nabla\phi_\tau(x),\alpha e_i+\beta e_j}^2+E_5,
    $$
    with $|E_5|\leq C(h^3\tau^2+h^4\tau^4)$. For the terms in $\sinh$ we will approximate the difference by a Taylor polynomial of order 3. Writing $g(s,t)=\phi_\tau(x+se_i+te_j)$ and computing the third degree Taylor polynomial of $g$ at $(0,0)$ we get
    \begin{multline}\label{P3:eq1}
        P_3g(s,t)= g(0,0)+\frac{\partial}{\partial s} g(0,0) s+\frac{\partial}{\partial t} g(0,0) t \\
        + \frac{1}{2}\frac{\partial^2}{\partial s^2} g(0,0) s^2+\frac{1}{2}\frac{\partial^2}{\partial t^2} g(0,0) t^2 +\frac{\partial^2}{\partial s \partial t} g(0,0) st \\
        + \frac{1}{6}\frac{\partial^3}{\partial s^3}g(0,0) s^3 +\frac{1}{6}\frac{\partial^3}{\partial s^3}g(0,0)t^3+\frac{1}{2}\frac{\partial^3}{\partial s^2 \partial t}g(0,0)s^2t+\frac{1}{2}\frac{\partial^3}{\partial s \partial t^2}g(0,0) st^2.
    \end{multline}
    The difference inside the $\sinh$ writes
    \begin{multline*}
        D^\alpha_iD^\beta_j\phi_\tau(x) =\beta D_i^\alpha \left(\phi_\tau(x+\beta h e_j)-\phi_\tau(x)\right) \\
        = \alpha \beta \left(\phi_\tau(x+\alpha h e_i+\beta h e_j)-\phi_\tau(x+\beta h e_j)-\phi_\tau(x+\alpha h e_i)+\phi_\tau(x)\right).
    \end{multline*}
    Applying \eqref{P3:eq1} to the above we get
    \begin{align*}
    D^\alpha_iD^\beta_j\phi_\tau(x)= h^2 e_i^*\nabla^2\phi_\tau(x) e_j + \alpha h^3 T_1 \phi_\tau(x)+\beta h^3T_2\phi_\tau(x)+E_6
    \end{align*}
    with $|E_6|\leq Ch^4\tau$, $T_1\phi_\tau(x)$, $T_2\phi_\tau(x)$ being the terms of order 3 so that $|T_i\phi_\tau(x)|\leq C\tau$. Hence
    $$
    \sinh(D^\alpha_iD^\beta_j\phi_\tau(x))=h^2 e_i^*\nabla^2\phi_\tau(x) e_j +\alpha h^3 T_1 \phi_\tau(x)+\beta h^3T_2\phi_\tau(x)+E_7,
    $$
    with $|E_7|\leq C(h^4\tau+h^6\tau^3)$.
    A multiplication of the above yields
    \begin{multline*}
        \frac{1}{h^4}\sum_{\alpha,\beta=\pm1} \alpha\beta\sinh(D^\alpha_iD^\beta_j\phi_\tau(x))\cosh\big(D_{i,j}^{\alpha,\beta}\phi_\tau(x)\big)f(x+h\alpha e_i)\overline{f(x+h\beta e_j)} \\
        = \mathbf{C}_{i,j}^{f,f}(x) +\sum_{\alpha,\beta=\pm1} h^{-1}\left[\beta  T_1 \phi_\tau(x)+\alpha T_2\phi_\tau(x)\right]f(x+h\alpha  e_i)\overline{f(x+\beta he_j)}\\
        +\sum_{\alpha,\beta=\pm}E_8f(x+h\alpha e_i)\overline{f(x+h\beta e_j)},
    \end{multline*}
    with $|E_8|\leq C(+\tau+\delta_0\tau^2+\delta_0^2\tau^3)\lesssim\tau^2+\delta_0^2\tau^3$ after simplification using that $\tau h\leq \delta_0$. 
    It remains to bound the third order term, which can be expressed in the form 
    \begin{align*}
        \sum_{\alpha,\beta=\pm1}& \left[\beta  T_1 \phi_\tau(x)+\alpha T_2\phi_\tau(x)\right]h^{-1}f(x+h\alpha e_i)\overline{f(x+\beta he_j)} \\
        &= \frac{1}{2}\partial_s\partial_t^2g(0,0)f(x+he_i)\overline{\left(h^{-1}D_j^s f(x)\right)} \\
        &+ \frac{1}{2}\partial_s\partial_t^2g(0,0)f(x-he_i)\overline{\left(h^{-1}D_j^s f(x)\right)} \\
        &+\frac{1}{2}\partial_s^2\partial_tg(0,0)\overline{f(x+he_j)}\left(h^{-1}D_i^s f(x)\right) \\
        &+ \frac{1}{2}\partial_s^2\partial_tg(0,0)\overline{f(x-he_j)}\left(h^{-1}D_i^s f(x)\right).
    \end{align*}
    Using Cauchy-Schwarz inequality, we get
    \begin{multline*}
         \sum_{x\in h\mathcal{V}}\sum_{i,j=1}^k\sum_{\alpha,\beta=\pm1} \alpha\beta T_3^{\alpha,\beta}\phi_\tau(x)h^{-1}f(x+h\alpha e_i)\overline{f(x+\beta he_j)} \\
         \leq C\tau^2\norm{f}^2 + C\sum_{j=1}^k\norm{h^{-1}D_j^s f}^2,
    \end{multline*}
    and thus the desired estimate.
\end{proof}

As for our next step, we freeze coefficients in the operators when acting on functions supported in sets of size $\lambda\tau^{-\frac{1}{2}}$ for $\lambda>0$ and $\tau >1$ sufficiently large, both of which needing to be determined when we finally prove the Carleman estimate \eqref{carleman:eq1}.

\begin{lemma}\label{lemma3:4} 
Assume that $1<\tau\leq \delta_0h^{-1}$ for a sufficiently small constant $\delta_0>0$.
Let $B\subset B_2\text{\textbackslash}\overline{B}_{\frac{1}{2}}$ be a ball of
radius $\lambda\tau^{-\frac{1}{2}}$ and $\overline{x}\in B$. 
Let $\phi_\tau$ be as in Lemma \ref{lemma:3.1}.
Set 
    \begin{align*}
        \overline{S_j}f(x) :=& h^{-2}\Delta_jf(x) + \frac{\scal{\nabla\phi_\tau(\overline{x}),e_j}^2}{2}\left(f(x+he_j)+f(x-he_j)\right),\\
        \overline{A_j}f(x) :=& -h^{-1}\scal{\nabla\phi_\tau(\overline{x}),e_j}D_j^sf(x), \\
        \overline{\mathbf{C}}_{i,j}^{f,f}(x) :=& h^{-2}e_i^*\nabla^2\phi_\tau(\overline{x})e_jD_i^sf(x)\overline{D_j^sf(x)}\\
        &+ \frac{1}{2}\sum_{\alpha,\beta=\pm} \alpha \beta e_i^*\nabla^2\phi_\tau(\overline{x})e_j\abs{\scal{\nabla\phi_\tau(\overline{x}),e_i+\alpha\beta e_j}}^2\\
        &\qquad\quad\times f(x+\alpha he_i)\overline{f(x+\beta he_j)}.
    \end{align*}
Then for $f\in \mathcal{C}_c^\infty(B_2\text{\textbackslash}\overline{B}_{\frac{1}{2}})$ with $\supp{(f)}\subset B$
        \begin{align*}
        \bigl|\|\Tilde{S}_{\phi_\tau} f\| &- \|\overline{S}_{\phi_\tau} f\|\bigr| \leq C\tau^{\frac{3}{2}}\lambda\norm{f},\\
          \bigl|\|\Tilde{A}_{\phi_\tau} f\| &- \|\overline{A}_{\phi_\tau} f\|\bigr|\leq C\tau^{\frac{1}{2}}\lambda\sum_{j=1}^k\norm{h^{-1}D_j^sf},\\
         \sum_{x\in(h\mathcal{V})} \sum_{i,j=1}^k\abs{\mathbf{C}_{i,j}^{f,f}-\overline{\mathbf{C}}_{i,j}^{f,f}} &\leq C\tau^{\frac{1}{2}}\lambda\sum_{j=1}^k \norm{h^{-1}D^s_jf}^2+ C\tau^{\frac{5}{2}}\lambda\norm{f}^2.
    \end{align*}
\end{lemma}

\begin{proof}
    The proof follows from the triangle inequality and the Mean Value Theorem. For the symmetric operators, we write
    \begin{multline*}
               \abs{\norm{\Tilde{S}_{\phi_\tau} f} - \norm{\overline{S}_{\phi_\tau} f}}\\
               \leq
        C\norm{\left(\scal{\nabla\phi_\tau(x),e_j}^2-\scal{\nabla\phi_\tau(\overline{x}),e_j}^2\right)\left(f(x+he_j)+f(x-he_j)\right)}.
        \end{multline*}
    Then the Mean Value Theorem implies that
    $$
    |\scal{\nabla\phi_\tau(x),e_j}^2-\scal{\nabla\phi_\tau(\overline{x}),e_j}^2|\leq C\tau^2|x-\bar x|
    $$
    where $C$ depends only on the first 2 derivatives of $\ffi$. As $|x-\bar x|\lesssim \tau^{-1/2}\lambda$ when 
    $x\in\supp f$, the result follows.
    
    For the antisymmetric  part, we do the same using
     $$
    |\scal{\nabla\phi_\tau(x),e_j}-\scal{\nabla\phi_\tau(\overline{x}),e_j}|\lesssim \tau|x-\bar x|\lesssim \tau^{1/2}\lambda.
    $$
For the commutator, we use
$$
|\nabla^2\phi_\tau(x)-\nabla^2\phi_\tau(\overline{x})|\lesssim \tau|x-\bar x|\lesssim \tau^{1/2}\lambda
$$
for the first part and, for the second part,
\begin{multline*}
\Bigl|e_i^*\nabla^2\phi_\tau(x)e_j\abs{\scal{\nabla\phi_\tau(x),\alpha e_i+\beta e_j}}^2\\
-e_i^*\nabla^2\phi_\tau(\overline{x})e_j\abs{\scal{\nabla\phi_\tau(\overline{x}),\alpha e_i+\beta e_j}}^2\Bigr|\\
\lesssim \tau^3|x-\bar x|\lesssim \tau^{5/2}\lambda,
\end{multline*}
and conclude with Cauchy-Schwarz.
\end{proof}

The next proposition is the last step we have to introduce before proving Theorem \ref{Carleman:1}.

\begin{proposition}\label{proposition:3.5}
There exists $\tau_1$ such that, for every $\tau_0>\tau_1$, there exist $C_\text{low}>0, c_0>\dfrac{1}{\tau_0}, h_0,\delta_0\in(0,1)$ such that for 
$\tau \in (\tau_0,\delta_0h^{-1})$, $h\in(0,h_0)$, 
 if $B\subset B_2\setminus\overline{B}_{\frac{1}{2}}$ is a ball of
radius $\lambda\tau^{-\frac{1}{2}}$ and $\overline{x}\in B$;
and $\overline{S_{\phi_\tau}}, \overline{A_{\phi_\tau}}$ and $\overline{\mathbf{C}}^{f,f}_{i,j}$ 
are as in the last lemma, then for every $f\in\mathcal{C}_c^\infty(B_2\setminus B_{\frac{1}{2}})$ supported in $B$,
we have
\begin{multline}
\label{eqproposition:3.5}
        \norm{\overline{S}_{\phi_\tau} f}^2+\norm{\overline{A}_{\phi_\tau} f}^2+c_0\tau \sum_{x\in(h\mathcal{V})} \sum_{i,j=1}^k\overline{\mathbf{C}}_{i,j}^{f,f}\\
        \geq C_\text{low}\left(\tau^4\norm{f}^2+\tau^2h^{-2}\sum_{j=1}^k\norm{D_j^sf}^2+h^{-4}\sum_{j=1}^k\norm{(D_j^s)^2f}^2\right).
    \end{multline}
\end{proposition}

\begin{proof}
To simplify notation, we write $\xi_j=2\pi\scal{\xi,e_j}_{\R^d}$.
    Since the operators have constant coefficients ($\overline{x}$ is fixed) we can perform a Fourier transform and use Parseval.
    Set
    $$
p_{s,j}(\xi):=2h^{-2}\left(\cos\left(h\xi_j\right)-1\right)+\scal{\nabla\phi_\tau(\overline{x}),e_j}^2\cos\left(h\xi_j\right),  
$$
and $p_s(\xi)=\dst\sum_{j=1}^k p_{s,j}(\xi)$ so that
$$
\norm{\overline{S}_\phi f}^2=\norm{p_s\hat f}^2.
$$
Next, set
$$
    p_{a,j}(\xi):= 2h^{-1}\scal{\nabla\phi_\tau(\overline{x}),e_j}\sin\left(h\xi_j\right),
$$
and $p_a(\xi)=\dst\sum_{j=1}^k p_{a,j}(\xi)$, then
$$
        \norm{\overline{A}_\phi f}^2=\norm{p_a\hat f}^2.
$$
 Finally, set $q(\xi)=q_1(\xi)+q_2(\xi)$ so that
 $$
c_0\tau \sum_{x\in(h\mathcal{V})} \sum_{i,j=1}^k\overline{\mathbf{C}}_{i,j}^{f,f}=c_0\scal{q(\xi)\hat{f},\hat{f}}=c_0\scal{q_1(\xi)\hat{f},\hat{f}}+c_0\scal{q_2(\xi)\hat{f},\hat{f}},
 $$
 with
 \begin{equation}\label{proposition:eq1}
     q_1(\xi)=4\tau h^{-2}\sum_{i,j=1}^k\sin(h\xi_i)\sin(h\xi_j)e_i^*\nabla^2\phi_\tau(\overline{x})e_j
 \end{equation}
 and
 \begin{equation}\label{proposition:eq5}
     q_2(\xi)=\tau \sum_{i,j=1}^k e_i^*\nabla^2\phi_\tau(\overline{x})e_j\left(\sum_{\eta=\pm1}\eta \abs{\scal{\nabla\phi_\tau(\overline{x}),e_i+\eta e_j}}^2\cos(h\xi_i-\eta h\xi_j)\right).
 \end{equation}

We have now introduced all the notation that will be needed in this proof. The first step is the following reformulation of the proposition:

\smallskip

\noindent{\bf Step 1.} {\em The proposition follows from the fact that, for $c_0$ small enough and for all $\xi$,}
\begin{equation}\label{claim:1}
    p_s^2(\xi)+p_a^2(\xi)+c_0q(\xi)\gtrsim \abs{\xi}_E^4+\tau^2\abs{\xi}_E^2+\tau^4.
\end{equation}

\smallskip

\begin{proof}[Proof of Step 1] First note that
\begin{equation}\label{claim:1bis}
\abs{\xi}_E^4+\tau^2\abs{\xi}_E^2+\tau^4
\gtrsim \left(\sum_{j=1}^k h^{-4}\sin^4(h\xi_j)+\tau^2\sum_{j=1}^k h^{-2}\sin^2(h\xi_j)+\tau^4 \right) .
\end{equation}

Computing the norm we get 
    \begin{multline*}
    \norm{\overline{S}_{\phi_\tau} f}^2+\norm{\overline{A}_{\phi_\tau} f}^2+c_0\tau \sum_{x\in(h\mathcal{V})} \sum_{i,j=1}^k\overline{\mathbf{C}}_{i,j}^{f,f}\\
    = \norm{p_s\hat f}^2+\norm{p_s\hat f}^2+c_0\scal{q(\xi)\hat{f},\hat{f}} \\
    =h^d\int_{[0,h^{-1}]^d} \left(p_s(\xi)^2+p_a(\xi)^2+c_0q(\xi)\right)|\hat{f}(\xi)|^2d\xi.
    \end{multline*}
    So, if \eqref{claim:1} is established, from \eqref{claim:1bis} we obtain
    \begin{align*}
        \norm{\overline{S}_{\phi_\tau} f}^2+&\norm{\overline{A}_{\phi_\tau} f}^2+c_0\tau \sum_{x\in(h\mathcal{V})} \sum_{i,j=1}^k\overline{\mathbf{C}}_{i,j}^{f,f} \\
        \gtrsim& h^d\tau^4\int_{[0,h^{-1}]^d} |\hat{f}(\xi)|^2d\xi +h^d\tau^2 \sum_{j=1}^k\int_{[0,h^{-1}]^d} h^{-2}\sin^2(h\xi_j)|\hat{f}(\xi)|^2d\xi \\
        &+h^d \sum_{j=1}^k\int_{[0,h^{-1}]^d} h^{-4}\sin^4(h\xi_j)|\hat{f}(\xi)|^2d\xi.
    \end{align*}
Finally by Parseval's identity the right hand side is
$$
    =\left(\tau^4\norm{f}^2+\tau^2h^{-2}\sum_{j=1}^k\norm{D_j^sf}^2+h^{-4}\sum_{j=1}^k\norm{(D_j^s)^2f}^2\right)
$$
and \eqref{eqproposition:3.5} follows.
\end{proof}

In order to prove \eqref{claim:1}, we will consider two cases depending on the value of $\abs{\xi}_E$.

\smallskip

\noindent{\bf Step 2.} {\em Proof of \eqref{claim:1} when $\abs{\xi}_E\geq \kappa\tau$. Here $\kappa$ is a constant that will be fixed in \eqref{eq:defkappa} below.}

\smallskip

Let us first bound $p_s^2$ from below. To do so, note that by our initial scaling of $h\Gamma$, see \eqref{scaling:1}, we have
$$
h \xi_j\leq 2\pi\sqrt{d}\|e_j\|_2 \|h\xi\|_\infty \leq c < \frac{\pi}{2}.
$$
Then since for $t\in[-c,c]$,
$$
\abs{\cos t -1}\gtrsim\abs{t}^2,
$$
we obtain that 
\begin{align*}
   \abs{\sum_{j=1}^k\frac{\cos(h\xi_j)-1}{h^2}}\gtrsim \frac{1}{h^2}\sum_{j=1}^k\abs{\xi_jh}^2\gtrsim \sum_{j=1}^k \abs{\xi_j}^2 \gtrsim \abs{\xi}_E^2.
\end{align*}
On the other hand
\begin{equation*}
    \abs{\sum_{j=1}^k \scal{\nabla\phi_\tau(\overline{x}),e_j}^2\cos(h\xi_j)} \lesssim \tau^2.
\end{equation*}
Combining the two, we get
\begin{align}
    p_s(\xi)^2 &\geq  \abs{2\sum_{j=1}^k\frac{\cos(h\xi_j)-1}{h^2}}^2-  \abs{\sum_{j=1}^k \scal{\nabla\phi_\tau(\overline{x}),e_j}^2\cos(h\xi_j)}^2 \notag \\
    &\geq C_{1}^2\abs{\xi}_E^4-C_2^2\tau^4 \gtrsim  \abs{\xi}_E^4. \notag 
\end{align}
Thus, if we chose 
\begin{equation}
    \label{eq:defkappa}
\kappa=\sqrt{2C_2/C_1}
\end{equation}
then, for $\abs{\xi}_E\geq \kappa \tau$,
\begin{align}
    p_s(\xi)^2
    &\geq C_{HF}\left(\abs{\xi}_E^4+\tau^2\abs{\xi}_E^2+\tau^4\right)  \label{proposition:eq6} 
\end{align}
where $C_{HF}>0$ is a constant independent of $\xi$. 

Also note that $c_0\tau q_1(\xi)$ from \eqref{proposition:eq1} is bounded, using Cauchy-Schwarz, by
\begin{multline*}
        4c_0\tau\sum_{i=1}^k \abs{h^{-1}\sin(h\xi_i)}\sum_{j=1}^k \abs{h^{-1}\sin(h\xi_j)}\abs{e_i^*\nabla^2\phi_\tau(\overline{x})e_j}\\
        \lesssim  c_0 \sum_{j=1}^k \tau^2 h^{-2}\sin^2(h\xi_j),  
\end{multline*}
while we can bound $c_0q_2(\xi)$ \eqref{proposition:eq5} by
\begin{equation*}
c_0\tau\sum_{i,j=1}^k \abs{e_i^*\nabla^2\phi_\tau(\overline{x})e_j} \sum_{\eta=\pm1}\abs{\scal{\nabla\phi_\tau(\overline{x}),e_i+\eta e_j}}^2 \lesssim c_0 \tau^4.
\end{equation*}

Hence by choosing $c_0$ to be small enough we ensure that the contribution of $c_0q(\xi)$ gets absorbed by the second and third term in \eqref{proposition:eq6}. In this case we discard $p_a^2(\xi)>0$.
The constraint $c_0\tau_0>1$ will then force us to take $\tau_1=1/c_0$ so that $\tau_0$ has only to be large enough.

\noindent{\bf Step 3.} {\em Proof of \eqref{claim:1} when $\abs{\xi}_E\leq \kappa\tau$ where $\kappa$ is the constant fixed in\eqref{eq:defkappa} above.}
\smallskip

In this region, we expand the symbols in $h\xi_j$ (noting that $h\abs{\xi_E}\leq \kappa\delta_0$ which is small for $\delta_0$ small) in a similar fashion as in Lemma \ref{lemma:3.3}. 
First for the symmetric part, using that 
$$
|\cos(h\xi_j)-1|\lesssim h^2|\xi_j|^2,
$$
and
$$
|\cos(h\xi_j)-1+\frac{h^2}{2}\xi_j^2| \lesssim h^4|\xi_j|^4.
$$
We obtain that 
\begin{align*}
    \bigl|p_{s,j}+\xi_j^2&-\scal{\nabla\phi_\tau(\overline{x}),e_j}^2\bigr| \\
    &=\abs{2h^{-2}\left(\cos\left(h\xi_j\right)-1+\frac{h^2\xi_j^2}{2}\right)+\scal{\nabla\phi_\tau(\overline{x}),e_j}^2\left(\cos\left(h\xi_j\right)-1\right)}  \\
    &\lesssim h^2|\xi_j|^4+\tau^2h^2|\xi_j|^2,
\end{align*}
by using the definition of $\phi_\tau$ (see Lemma \ref{lemma:3.1}).
Finally, using the equivalence of norms, we infer that 
\begin{align*}
    \abs{p_s(\xi)-\left(\abs{\nabla_E\phi_\tau(\overline{x})}^2-\abs{\xi_E}^2\right)}&=\abs{\sum_{j=1}^k p_{s,j}+\xi_j^2-\scal{\nabla\phi_\tau(\overline{x}),e_j}^2}\\
    &\lesssim \sum_{j=1}^k \left(h^{2}\abs{\xi_j}^4+\tau^2h^2\abs{\xi_j}^2\right) \\
    &\lesssim h^2\abs{\xi}_E^4+\tau^2h^2\abs{\xi}_E^2\lesssim h^2\tau^4 \lesssim \delta_0^2\tau^2.
\end{align*}

For the anti-symmetric part, using that
$$
\left|\sin(h\xi_j)-h\xi_j\right|\lesssim h^3|\xi_j|^3,
$$
we get 
\begin{align*}
\abs{p_{a,j}(\xi)-2\scal{\nabla\phi_\tau(\overline{x}),e_j}\xi_j}&=\abs{2h^{-1}\scal{\nabla\phi_\tau(\overline{x}),e_j}(\sin\left(h\xi_j\right)-h\xi_j)}\\
&\lesssim \tau h^2 |\xi_j|^3.
\end{align*}
Hence, after summing over all $j$'s and using the equivalence of norms, we get
\begin{align*}
    \abs{p_a(\xi)-2\scal{\nabla_E\phi_\tau(\overline{x}),\xi_E}_{\R^k}}&\lesssim \tau \sum_{j=1}^kh^2\abs{\xi_j}^3 \\
    &\lesssim \tau h^2 \abs{\xi}_E^3 \lesssim h^2\tau^4 \lesssim \delta_0^2\tau^2.
\end{align*}

We may thus write
$$
p_s(\xi)=\abs{\nabla_E\phi_\tau(\overline{x})}^2-\abs{\xi_E}^2+E_1, \quad \mbox{and} \quad p_a(\xi)=2\scal{\nabla_E\phi_\tau(\overline{x}),\xi_E}_{\R^k}+E_2,
$$
with $|E_1|,|E_2|\lesssim \delta_0^2\tau^2$. As $\abs{\nabla_E\phi_\tau(\overline{x})}$, $\abs{\xi_E}\lesssim\tau$, thus 
$\abs{\scal{\nabla_E\phi_\tau(\overline{x}),\xi_E}_{\R^k}}\lesssim\tau^2$,  it follows that 
\begin{equation}\label{proposition:eq7}
p_s^2(\xi)+p_a^2(\xi)=\left(\abs{\nabla_E\phi_\tau(\overline{x})}^2-\abs{\xi}_E^2\right)^2+4\scal{\nabla_E\phi_\tau(\overline{x}),\xi_E}_{\R^k}^2+E_3,
\end{equation}
with $|E_3|\lesssim\delta_0^4\tau^4+\delta_0^2\tau^4\lesssim \delta_0^2\tau^4$.

Now let
\begin{equation*}
    \mathcal{S}_\tau:=\left\{\xi\,:\ \abs{\xi}_E^2=\abs{\nabla_E\phi_\tau(\overline{x})}^2\right\}
    \quad,\quad \mathcal{F}:=\left\{\xi\,:\ \scal{\nabla_E\phi_\tau(\overline{x}),\xi_E}_{\R^k}=0\right\}
\end{equation*}
and
\begin{equation*}
\mathcal{C}_\tau=\mathcal{S}_\tau\cap\mathcal{F}
\end{equation*}
denote the joint characteristic sets of the symmetric and antisymmetric parts of the operator. Further define
\begin{equation*}
    \mathcal{N}_{\tau,\mathcal{C}}:=\left\{\xi\in \mathbb{T}^d_h:\mbox{dist}(\xi_E,\mathcal{C}_\tau)<\gamma_0\tau\right\}
\end{equation*}
to be a $\gamma_0\tau$ neighbourhood of the joint characteristic set $\mathcal{C}_\tau$ with $\gamma_0>0$ small.
In order to prove \eqref{claim:1}, we split the study into two cases depending on $\xi$.

\smallskip

\noindent{\bf First Case.} {\sl We assume that $\xi$ is outside of $\mathcal{N}_{\tau,\mathcal{C}}$.}

\smallskip

First note that it is enough to prove that
\begin{equation}\label{proposition:eq2}
    p_s^2(\xi)+p_a^2(\xi)\gtrsim \tau^4
\end{equation}
since $\tau^4 \gtrsim(\tau^4+\tau^2\abs{\xi}_E^2+\abs{\xi}_E^4)$ under this hypothesis $\abs{\xi}_E\lesssim\tau$.
We then deduce \eqref{claim:1}  by absorbing $c_0q(\xi)$ into $p_s^2(\xi)+p_a^2(\xi)$ when $c_0$ is small enough (thus $\tau_1$ large enough), in the same manner as at the end of Step 2.

We now prove \eqref{proposition:eq2} with the help of \eqref{proposition:eq7}.

Take $\xi_{\mathcal{C}}\in\mathcal{C}_\tau$ to be the closest point to $\xi$ in the charasterictic set, for which $|\xi-\xi_\mathcal{C}|\geq \gamma_0\tau$. In other words,
$$
\xi_{\mathcal{C}}=\arg\min_{\tilde\xi\in\mathcal{C}_\tau}|\xi-\tilde\xi|.
$$
Now we estimate the leading order approximation in \eqref{proposition:eq7}. As $\xi_{\mathcal{C}}\in\mathcal{C}_\tau$, we have
\begin{align*}
    \left|\abs{\nabla_E\phi_\tau(\overline{x})}^2-\abs{\xi}_E^2\right|^2+4\scal{\nabla_E\phi_\tau(\overline{x}),\xi_E}^2
    = \bigl( \abs{\xi_\mathcal{C}}_E^2-\abs{\xi}_E^2\bigr)^2+4\scal{\nabla_E\phi_\tau(\overline{x}),\xi_E}^2.
\end{align*}

By Lemma \ref{lemma:3.1}, we have that there exist $a,b>0$ such that 
$$
a\tau\leq\abs{\nabla_E\phi_\tau(\overline{x})}\leq b\tau.
$$ 
Let $\rho\in[a,b]$ be such that 
$$
\abs{\nabla_E\phi_\tau(\overline{x})}=\rho\tau. 
$$

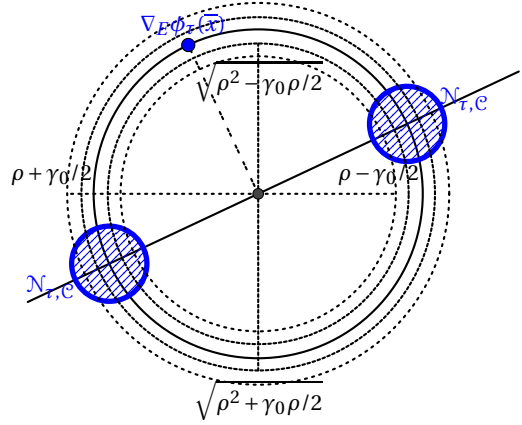
\begin{figure}
    \centering
\begin{tikzpicture}[line cap=round,line join=round,>=triangle 45,x=0.5cm,y=0.5cm]
\clip(-7,-6) rectangle (8,7);
\draw [line width=0.8pt] (0.,0.) circle (2.187532856896097cm);
\draw [line width=0.8pt,domain=-6.127770099493457:6.92849824037501] plot(\x,{(-0.--1.86*\x)/3.96});
\draw [line width=2.pt,color=qqqqff,fill=qqqqff,pattern=north east lines,pattern color=qqqqff] (-3.96,-1.86) circle (0.5cm);
\draw [line width=2.pt,color=qqqqff,fill=qqqqff,pattern=north east lines,pattern color=qqqqff] (3.96,1.86) circle (0.5cm);
\draw [line width=0.8pt,dash pattern=on 1pt off 1pt] (0.,0.) circle (2.cm);
\draw [line width=0.8pt,dash pattern=on 1pt off 1pt] (0.,0.) circle (2.35cm);
\draw [line width=0.8pt,dotted] (0.,0.) circle (1.83cm);
\draw [line width=0.8pt,dotted] (0.,0.) circle (2.539547203735343cm);
\draw [line width=0.8pt,dash pattern=on 1pt off 1pt on 1pt off 4pt] (-1.86,3.96)-- (0.,0.);
\draw [line width=0.8pt,dash pattern=on 1pt off 1pt] (0.,0.) -- (0,4);
\draw [line width=0.8pt,dash pattern=on 1pt off 1pt] (0.,0.) -- (0,-4.7);
\draw [line width=0.8pt,dotted] (0.,0.) -- (3.65,0);
\draw [line width=0.8pt,dotted] (0.,0.) -- (-5.08,0);
\begin{scriptsize}
\draw [fill=uuuuuu] (0.,0.) circle (2.0pt);
\draw[color=black] (-10.68491193814343,-4.7116180503326);
\draw [fill=qqqqff] (-1.86,3.96) circle (2.5pt);
\node at (0,3) {$\sqrt{\rho^2-\gamma_0\rho/2}$};
\node at (0,-5.5) {$\sqrt{\rho^2+\gamma_0\rho/2}$}; 
\node at (3.2,0.5) {$\rho-\gamma_0/2$};
\node at (-5.5,0.5) {$\rho+\gamma_0/2$};
\node[color=qqqqff] at (-5.5,-2.5) {$\mathcal{N}_{\tau,\mathcal{C}}$};
\node[color=qqqqff] at (5.5,2.5)  {$\mathcal{N}_{\tau,\mathcal{C}}$};
\node[color=qqqqff] at (-2,4.5)  {$\nabla_E\phi_\tau(\overline{x})$};
\end{scriptsize}
\end{tikzpicture}
\caption{Sketch of $\mathcal{N}_{\tau,\mathcal{C}}$ (hatched area) in dimension 2 where $\tau=1$.}
\end{figure}

Assume first that $\mathrm{dist}(\xi_E,\mathcal{S}_\tau)<\gamma_0\tau/2$, that is
$(\rho-\gamma_0/2)\tau< \abs{\xi}_E< (\rho+\gamma_0/2)\tau$, so that $\xi \notin \mathcal{N}_{\tau,\mathcal{C}}$ implies that
$\xi\notin \mathcal{F}$, that is $\scal{\nabla_E\phi_\tau(\overline{x}),\xi_E}_{\R^k}\not=0$. Then
$$
\abs{\scal{\nabla_E\phi_\tau(\overline{x}),\xi_E}_{\R^k}}^2=\abs{\nabla_E\phi_\tau(\overline{x})}^2 \ \mbox{dist}(\xi_E,\mathcal{F})^2  \gtrsim \gamma_0^2\tau^4
$$
where the implied constant depends on $\rho$ and can thus be chosen as an absolute constant.

On the other hand, if $\mathrm{dist}(\xi_E,\mathcal{S}_\tau)\geq \gamma_0\tau/2$. We have two cases:

-- Either $|\xi|_E\leq (\rho-\gamma_0/2)\tau$, but then
$$
|\xi|_E^2\leq (\rho^2-\gamma_0\rho+\gamma_0^2/4)\tau^2\leq (\rho^2-\gamma_0\rho/2)\tau^2=|\xi_{\mathcal{C}}|_E^2-\frac{\gamma_0\rho}{2}\tau^2,
$$
provided $\gamma_0\leq 2a\leq 2\rho$. 

-- Otherwise $|\xi|_E\geq (\rho+\gamma_0/2)\tau$ and then
$$
|\xi|_E^2\geq (\rho^2+\gamma_0\rho+\gamma_0^2/4)\tau^2\geq (\rho^2+\gamma_0\rho/2)\tau^2=|\xi_{\mathcal{C}}|_E^2+\frac{\gamma_0\rho}{2}\tau^2.
$$
In summary, we now have
$$
\left(\abs{\nabla_E\phi_\tau(\overline{x})}^2-\abs{\xi}_E^2\right)^2=\left(\abs{\xi_\mathcal{C}}^2-\abs{\xi}_E^2\right)^2 \geq \rho^2\gamma_0^2\tau^4/4
\gtrsim \gamma_0^2\tau^4.
$$

Combining the above yields that 
$$
\left(\abs{\nabla_E\phi_\tau(\overline{x})}^2-\abs{\xi}_E^2\right)^2+4\scal{\nabla_E\phi_\tau(\overline{x}),\xi_E}_{\R^k}^2  \gtrsim \gamma_0^2\tau^4.
$$
Therefore, \eqref{proposition:eq7} gives
$$ 
p_s^2(\xi)+p_a^2(\xi) \geq C_1\gamma_0^2\tau^4-C_2\delta_0^2\tau^4 \gtrsim \tau^4,
$$
provided $\delta_0$ has been chosen small enough with respect to $\gamma_0$.

\smallskip

\noindent{\bf Second Case.} {\sl We assume that $\xi\in\mathcal{N}_{\tau,\mathcal{C}}$.}

\smallskip

Here we will use that $|e_i^*\nabla^2\phi_\tau(\overline{x})e_j|\lesssim\tau$, $|\xi_j|\lesssim\tau$,
$$
|\cos(h\xi_j)-1|\lesssim |h\xi_j|^2\lesssim h^2\tau^2,
\quad\mbox{and}\quad
\left|\sin(h\xi_j)-h\xi_j\right|\lesssim h^3|\xi_j|^3\lesssim h^3\tau^3.
$$

In order to prove \eqref{claim:1}, we use that $p_s^2(\xi)+p_a^2(\xi)\geq 0$ and show that the inequality is already valid for $q$ in this case.
Recall that $q(\xi)=q_1(\xi)+q_2(\xi)$ where $q_1,q_2$ are defined in \eqref{proposition:eq1} and \eqref{proposition:eq5} respectively.
On the one hand, 
writing $\sin(h\xi_i)=\sin(h\xi_i)-h\xi_i+h\xi_i$
and the same for $\sin(h\xi_j)$,
we write the first one as
\begin{eqnarray*}
    q_1(\xi)&=&
    4\tau h^{-2}\sum_{i,j=1}^k\bigl(\sin(h\xi_i)\bigr)\bigl(\sin(h\xi_j)\bigr)e_i^*\nabla^2\phi_\tau(\overline{x})e_j\\
    &=&\sum_{i,j}^k\left(e_i^*\nabla^2\phi_\tau(\overline{x})e_j\right)\xi_i\xi_j\\
    &&+4\tau h^{-2}\sum_{i,j=1}^k\bigl(\sin(h\xi_i)-h\xi_i\bigr)\bigl(\sin(h\xi_j)-h\xi_j\bigr)e_i^*\nabla^2\phi_\tau(\overline{x})e_j\\
    &&+4\tau h^{-1}\sum_{i,j=1}^k\xi_i\bigl(\sin(h\xi_j)-h\xi_j\bigr)e_i^*\nabla^2\phi_\tau(\overline{x})e_j\\
    &&+4\tau h^{-1}\sum_{i,j=1}^k\bigl(\sin(h\xi_i)-h\xi_i\bigr)\xi_je_i^*\nabla^2\phi_\tau(\overline{x})e_j\\
    &=&4\tau\sum_{i,j}^k\left(e_i^*\nabla^2\phi_\tau(\overline{x})e_j\right)\xi_i\xi_j+ \eps_1=4\tau\xi_E^*\nabla^2_E\phi_\tau(\overline{x})\xi_E+ \eps_1
\end{eqnarray*}
with $\eps_1\lesssim h^4\tau^8+h^2\tau^6\lesssim (\delta_0^4+\delta_0^2)\tau^4\lesssim \delta_0^2\tau^4$. Note that we are in the region where $|\xi|_E\lesssim\tau$.

On the other hand, recall \eqref{proposition:eq5}: 
$$
   q_2(\xi)=\tau \sum_{i,j=1}^k e_i^*\nabla^2\phi_\tau(\overline{x})e_j\left(\sum_{\eta=\pm1}\eta \abs{\scal{\nabla\phi_\tau(\overline{x}),e_i+\eta e_j}}^2\cos(h\xi_i-\eta h\xi_j)\right).\
$$
We first expand $\abs{\scal{\nabla\phi_\tau(\overline{x}),e_i+\eta e_j}}^2$ and rewrite the sum in $\eta$ as
$$
\begin{aligned}
&\bigl(\abs{\scal{\nabla\phi_\tau(\overline{x}),e_i}}^2+\abs{\scal{\nabla\phi_\tau(\overline{x}),e_j}}^2\bigr)\bigl(\cos(h\xi_i-h\xi_j)-\cos(h\xi_i+ h\xi_j)\bigr)\\
&\qquad+2\scal{\nabla\phi_\tau(\overline{x}),e_i}\scal{\nabla\phi_\tau(\overline{x}),e_j}\bigl(\cos(h\xi_i-h\xi_j)+\cos(h\xi_i+ h\xi_j)\bigr)\\
=&2\bigl(\abs{\scal{\nabla\phi_\tau(\overline{x}),e_i}}^2+\abs{\scal{\nabla\phi_\tau(\overline{x}),e_j}}^2\bigr)\sin h\xi_i\sin h\xi_j\\
&\qquad+4\scal{\nabla\phi_\tau(\overline{x}),e_i}\scal{\nabla\phi_\tau(\overline{x}),e_j}\cos h\xi_i\cos h\xi_j\\
=&4\scal{\nabla\phi_\tau(\overline{x}),e_i}\scal{\nabla\phi_\tau(\overline{x}),e_j}+O(h^2\tau^4).
\end{aligned}
$$
Summing those identities, we finally obtain
\begin{eqnarray*}
 q_2(\xi)&=&4\tau\sum_{i,j=1}^k \left(e_i^*\nabla^2\phi_\tau(\overline{x})e_j\right)\scal{\nabla\phi_\tau(\overline{x}),e_i}\scal{\nabla\phi_\tau(\overline{x}),e_j}+\eps_2\\
 &=&4\tau\nabla_E\phi_\tau(\overline{x})^*\nabla^2_E\phi_\tau(\overline{x})\nabla_E\phi(\overline{x})+\eps_2
\end{eqnarray*}
with $|\eps_2|\lesssim h^2\tau^6 \lesssim\delta_0^2\tau^4$.

In summary, we obtain that
\begin{equation*}
    q(\xi)=4\tau\left(\nabla_E\phi_\tau(\overline{x})^*\nabla^2_E\phi_\tau(\overline{x})\nabla_E\phi(\overline{x})+ \xi_E^*\nabla^2_E\phi_\tau(\overline{x})\xi_E\right) + \eps_3
\end{equation*}
with $|\eps_3|\lesssim \delta_0^2\tau^4$.

Now by applying Lemma \ref{lemma:3.1}, we infer that for $\overline{x}\in B_2\mbox{\textbackslash}B_{1/2}$ and for $\xi$ in the neighborhood of the characteristic set, we get that for $\delta_0$ small enough, there exist a constant $c_{cf}>0$ such that
\begin{equation*}
    q(\xi)\geq C\tau^4 \geq c_{cf}(\tau^4+\tau^2\abs{\xi}_E^2+\abs{\xi}_E^4),
\end{equation*}
finishing the proof.

\end{proof}

We are now ready to prove our Carleman estimate.
\begin{proof}[Proof of Theorem \ref{Carleman:1}]
Recall that we want to prove \eqref{carleman:eq1}:
   \begin{multline*}
        \tau^3 \norm{e^{ \phi_\tau} u}^2_{\ell^2(h\mathcal{V})} + \tau \norm{e^{\phi_\tau} h^{-1}D_s u}^2_{\ell^2(h\mathcal{V})} + \tau^{-1} \norm{e^{ \phi_\tau} h^{-2}D^2_s u}^2_{\ell^2(h\mathcal{V})} \\
        \leq C \norm{e^{\phi_\tau} g}^2_{\ell^2(h\mathcal{V})}.
    \end{multline*}
We first reformulate this in terms of $f=e^{\phi_\tau} u$ and 
$$
L_{\phi_\tau} f(x):= h^{-2}e^{\phi_\tau}\Delta_{h\Gamma} e^{-\phi_\tau}f(x). 
$$

\smallskip
    
\noindent{\bf Step 1.} {\em It is enough to prove that}
\begin{equation}\label{estimates:1}
        \tau^3\norm{f}^2+\tau\norm{h^{-1}D_sf}^2+\tau^{-1}\norm{h^{-2}D^2_sf}^2 \leq C'\norm{L_{\phi_\tau} f}^2.
    \end{equation}
    
\smallskip

Note that \eqref{estimates:1} is equivalent to 
\begin{equation}\label{estimates:1bis}
        \tau^{3/2}\norm{f}+\tau^{1/2}\norm{h^{-1}D_sf}+\tau^{-1/2}\norm{h^{-2}D^2_sf} \lesssim\norm{L_{\phi_\tau} f}
    \end{equation}
with a different constant.
   
\smallskip

Let $x\in B_2\setminus B_{1/2}$. Write
\begin{eqnarray*}
  e^{\phi_\tau(x)}D_j^su(x)&=&e^{\phi_\tau(x)}u(x+he_j)-e^{\phi_\tau(x)}u(x-he_j)\\
  &=&e^{\phi_\tau(x)-\phi_\tau(x+he_j)} f(x+he_j)-e^{\phi_\tau(x)-\phi_\tau(x-he_j)} f(x-he_j)\\
  &=&D_j^sf(x)+\bigl(e^{\phi_\tau(x)-\phi_\tau(x+he_j)}-1\bigr)f(x+he_j)\\
  &&-\bigl(e^{\phi_\tau(x)-\phi_\tau(x-he_j)}-1\bigr)f(x-he_j).
\end{eqnarray*}
Now note that 
\begin{align*}
|\phi_\tau(x)-&\phi_\tau(x+he_j)|= \tau \abs{\ffi(|x+he_j|_\Gamma)-\ffi(|x|_\Gamma)}\\
&\leq \tau\bigl||x+he_j|_\Gamma-|x|_\Gamma\bigr| \sup\{|\ffi'(t)|\,:\ t\in[|x|_\Gamma,|x+he_j|_\Gamma]\}.
\end{align*}
Note that, on the one hand $\bigl||x+he_j|_\Gamma-|x|_\Gamma\bigr|\leq |h|\|e_i\|_\Gamma\lesssim h$ and on the other hand, there are $0<a<b<+\infty$
such that, if $x\in B_2\setminus B_{1/2}$ then $[|x|_\Gamma,|x+he_j|_\Gamma]\subset[a,b]$. But then, as $\ffi$ is of class $\cc^1$
on $(0,+\infty)$, we obtain that 
$$
\sup\{|\ffi'(t)|\,:\ t\in[|x|_\Gamma,|x+he_j|_\Gamma]\}\lesssim 1.
$$
It follows that
$$
|\phi_\tau(x)-\phi_\tau(x+he_j)|\lesssim h\tau.
$$
Further $h\tau\leq \delta_0\leq 1$ so that there is some $A>0$ such that $|\phi_\tau(x)-\phi_\tau(x+he_j)|\leq A$.
As $|e^u-1|\leq e^B|u|$ if $|u|\leq B$ we obtain that
$$
|e^{\phi_\tau(x)-\phi_\tau(x+he_j)}-1|\lesssim h\tau.
$$
We can argue the same way for $|e^{\phi_\tau(x)-\phi_\tau(x-he_j)}-1|$. Summing over $j$ and taking norms in $x$, we can thus
absorb the two corresponding terms into $\tau^3\norm{f}^2$ by taking a larger constant in the right hand side of \eqref{estimates:1}.
A similar argument allows also to absorb the extra terms comming from $ \norm{e^{ \phi_\tau} h^{-2}D^2_s u}^2_{\ell^2(h\mathcal{V})}$
into the two first terms.

\smallskip
    
\noindent{\bf Step 2.} {\em Localization.}

\smallskip

We claim that it suffices to show the estimate \eqref{estimates:1bis} for the functions $f_n$ of Lemma \ref{lemma:3.2}, provided the parameter $\lambda$ is chosen properly.
Indeed, if the estimate is proven for $f_n$, as $f=\sum f_n$ we would get 
$$\begin{aligned}
    \tau^\frac{3}{2}\norm{f}+&\tau^\frac{1}{2}\norm{h^{-1}D_sf}+\tau^\frac{-1}{2}\norm{h^{-2}D^2_sf}\\
    \leq& \sum_n\bigl(\tau^\frac{3}{2}\norm{f_n}+\tau^\frac{1}{2}\norm{h^{-1}D_sf_n}+\tau^\frac{-1}{2}\norm{h^{-2}D^2_sf_n}
    \lesssim\sum_{n}\norm{L_{\phi_\tau} f_n} \\
    \leq& C_{\text{loc}}\norm{L_{\phi_\tau} f} + C_{loc}\tau^\frac{1}{2}\lambda^{-1}\sum_{j=1}^k\norm{h^{-1}D_j^s f}\\
    &+  C_{\text{loc}}\left(\tau\lambda^{-2}+\tau^\frac{3}{2}\lambda^{-1}+\tau^{\frac{5}{2}}h\lambda^{-1}\right)\norm{f},
 \end{aligned}$$
for some constant $C_{\text{loc}}\geq 1$ coming from Lemma \ref{lemma:3.2}.

Now choosing $\lambda=10\sqrt{k} C_{\text{loc}}$
and recalling that $\tau h\leq \delta_0$ for some $\delta_0\in(0,1)$, we end up with
\begin{align*}
     \tau^\frac{3}{2}\norm{f}+\tau^\frac{1}{2}&\norm{h^{-1}D_sf}+\tau^\frac{-1}{2}\norm{h^{-2}D^2_sf} \\ 
     \leq& C\norm{L_{\phi_\tau} f} + \frac{1}{10\sqrt{k}}\tau^\frac{1}{2}\sum_{j=1}^k\norm{h^{-1}D_j^s f}
    + \frac{3}{10}\tau^\frac{3}{2}\norm{f}\\
    \leq& C\norm{L_{\phi_\tau} f} + \frac{1}{10}\tau^\frac{1}{2}\norm{h^{-1}D_sf}+ \frac{3}{10}\tau^\frac{3}{2}\norm{f}
\end{align*}
The two last terms are then absorbed into the left hand side to obtain the desired estimate for $f$.

\smallskip

\noindent{\bf Step 3.} {\em Proving the Theorem for localized functions.}

\smallskip

 We seek to prove \eqref{estimates:1} for $f_n=f\psi_n$ with $\supp{f}\subset B_2\text{\textbackslash}B_{\frac{1}{2}}$, $\psi_n$ as in Lemma \ref{lemma:3.2}. From Proposition \ref{proposition:3.5}, there is a $c_0\in(0,1)$ and $\tau_0>1$
with $\tau_0c_0\geq1$. We first write
\begin{align*}
    \tau\norm{L_{\phi_\tau} f_n}^2&=\tau\norm{S_{\phi_\tau} f_n}^2+\tau\norm{A_{\phi_\tau} f_n}^2+\tau \scal{[S_{\phi_\tau},A_{\phi_\tau}]f_n,f_n} \\
    &\geq \norm{S_{\phi_\tau} f_n}^2+\norm{A_{\phi_\tau} f_n}^2+\tau c_0\scal{[S_{\phi_\tau},A_{\phi_\tau}]f_n,f_n}
\end{align*}
where we use that if $\tau c_0\geq 1$, then, for every $a,b\geq0$ and every $\alpha\in[-ab,ab]$
$$
\tau a^2+\tau b^2+2\tau\alpha\geq a^2+b^2+2\tau c_0\alpha.
$$
Using Lemma \ref{lemma:3.3} and taking into account that $0<\delta_0<1<\tau$, it follows that
$$
\tau\norm{L_{\phi_\tau} f_n}^2
\geq \norm{\Tilde{S}_{\phi_\tau} f_n}^2+\norm{\Tilde{A}_{\phi_\tau} f_n}^2+\tau c_0\sum_{x\in(h\mathcal{V})} \sum_{i,j=1}^k\mathbf{C}_{i,j}^{f_n,f_n}-E_1,
$$
with
\begin{equation}
        0\leq E_1\leq C\left(\delta_0^2\tau^4+\tau^3\right)\norm{f_n}^2+C\tau\sum_{j=1}^k\norm{h^{-1}D_j^sf_n}^2.\\
\label{estimates:2}
\end{equation}
    Now exploiting the bound from Lemma \ref{lemma3:4} we get
\begin{align*}
\tau\norm{L_{\phi_\tau} f_n}^2\geq       
&\norm{\Tilde{S}_{\phi_\tau} f_k}^2+\norm{\Tilde{A}_{\phi_\tau} f_n}^2+\tau c_0\sum_{x\in(h\mathcal{V})} \sum_{i,j=1}^k\mathbf{C}_{i,j}^{f_n,f_n}-E_1 \\
        &\geq \norm{\overline{S}_{\phi_\tau} f_n}^2+\norm{\overline{A}_{\phi_\tau} f_n}^2+\tau c_0\sum_{x\in(h\mathcal{V})} \sum_{i,j=1}^k\overline{\mathbf{C}}_{i,j}^{f_n,f_n}-E_1-E_2,
\end{align*}
where
\begin{eqnarray*}
    0\leq E_2&\leq& C\left(\tau^3\lambda^2+\tau^\frac{7}{2}\lambda\right)\norm{f_n}^2+C\left(\tau\lambda
    ^2+\tau^\frac{3}{2}\lambda\right)\sum_{j=1}^k\norm{h^{-1}D_j^sf_n}^2\\
    &\leq&C\left(\tau^3+\tau^\frac{7}{2}\right)\norm{f_n}^2+C\left(\tau+\tau^\frac{3}{2}\right)\sum_{j=1}^k\norm{h^{-1}D_j^sf_n}^2,
\end{eqnarray*}
since $\lambda$ has been chosen in Step 2.

Finally, using Proposition \ref{proposition:3.5} we infer that
\begin{align*}
\tau&\norm{L_{\phi_\tau} f_n}^2\geq    
\norm{\overline{S}_{\phi_\tau} f_n}^2+\norm{\overline{A}_{\phi_\tau} f_n}^2+\tau c_0\sum_{x\in(h\mathcal{V})} \sum_{i,j=1}^k\overline{\mathbf{C}}_{i,j}^{f_n,f_n}-E_1-E_2 \\
    &\geq C_{\text{low}}\left(\tau^4\norm{f_n}^2+\tau^2h^{-2}\sum_{j=1}^k\norm{D_j^sf_n}^2+h^{-4}\sum_{j=1}^k\norm{(D_j^s)^2f_n}^2\right)-E_1-E_2. 
\end{align*}

Now, if we chose $\delta_0$ small enough and $\tau$ large enough, we have
\begin{align*}
0\leq E_1+E_2&\leq C\left(\delta_0^2+\tau^{-1}+\tau^{-\frac{1}{2}}\right)\tau^4\norm{f_n}^2\\
&\qquad+C\left(\tau^{-1}+\tau^{-\frac{1}{2}}\right)\tau^2 h^{-2}\sum_{j=1}^k\norm{D_j^sf_n}^2\\
\leq&\frac{C_{\text{low}}}{2}\left(\tau^4\norm{f_n}^2+\tau^2h^{-2}\sum_{j=1}^k\norm{D_j^sf_n}^2\right).
\end{align*}

We thus absorb the error terms and obtain
$$
\tau\norm{L_{\phi_\tau} f_n}^2\geq    
 C_{\text{low}}\left(\tau^4\norm{f_n}^2+\tau^2h^{-2}\sum_{j=1}^k\norm{D_j^sf_n}^2+h^{-4}\sum_{j=1}^k\norm{(D_j^s)^2f_n}^2\right)
$$
and dividing by $\tau>\tau_0$ finishes the proof.
\end{proof}
\section{The Three Balls Inequality}

In this section, we are still working with a 1-point periodic graph $h\Gamma$, for which we have the Carleman estimate \eqref{Carleman:1} proven in the previous section. We will establish the Three Balls Inequality
for Schr\"odinger operators on such graphs.
For future use, we will slightly enlarge the class of operators and,
instead of local potentials $V(x)f(x)$ consider a larger class:
Fix an integer $L$ and consider 
$$
\mathcal{P}_L=\left\{\sum_{j=1}^k\alpha_je_j\,:, \alpha_j\in\Z,\ \sum_{j=1}^k|\alpha_j|=L\right\}
$$
so that, for any $x,y\in h\mathcal{V}$
there is an $e\in\mathcal{P}_L$ such that $y=x+he$
if and only if there is a simple path of length at most $L$ between $x$ and $y$. For $L=0$, $\mathcal{P}_L=\{0\}$.

We consider operators of the form 
\begin{eqnarray}
    H_{h\Gamma}^L u(x)&=&  h^{-2}\Delta_{h\Gamma} u(x)+h^{-1}\sum_{j=1}^k B_j(x)D_s^ju(x)\notag\\
&&+\sum_{e\in\mathcal{P}_L} V_e(x+he)u(x+he)\label{eq:formhextrig}
\end{eqnarray}
with $\max(\norm{V_e}_{L^\infty},\norm{B_j}_{L^\infty})<+\infty$.

Note that when $L=0$ and $B_j=0$, $H_{h\Gamma}^0f(x)=h^{-2}\Delta_{h\Gamma} f(x)+V_0f(x)$ 
is a classical Schr\"odinger operator.

Our goal is to prove the following theorem.
\begin{theorem}[Three Balls Inequality]\label{TBT:4.2}
Let $\Gamma$ be a 1-point periodic graph $L\geq 0$
and $H_{h\Gamma}^L$ be an operator of the above form.
Then,
\begin{enumerate}
    \item there exist constants $h_0, \delta_0 \in (0,1)$, $c_1,c_2>0$ and $\tau_0>1$ such that for all $\tau\in(\tau_0,\delta_0h^{-1})$, $h\in(0,h_0)$ and $u:h\mathcal{V}\mapsto\R$ with $H^L_{h\Gamma}u=0$ on $B_4$ we have
    \begin{equation}\label{TCT:4.1eq1}
        \norm{u}_{\ell^2(B_1)} \leq C\left(e^{c_1\tau}\norm{u}_{\ell^2(B_{1/2})}+e^{-c_2\tau}\norm{u}_{\ell^2(B_2)}\right).
    \end{equation}
\item there exist $\alpha \in (0,1)$, $c_0>0$, $h_0\in(0,1)$ and $C>1$ such that for $h\in(0,h_0)$ and $u:h\mathcal{V}\to\R$ with $H^L_{h\Gamma} u=0$ in $B_4$ we have 
    \begin{equation}\label{TCT:4.2eq2}
        \norm{u}_{\ell^2(B_1)} \leq D\left(\norm{u}_{\ell^2(B_{1/2})}^\alpha\norm{u}_{\ell^2(B_2)}^{1-\alpha} + e^{-c_0h^{-1}}\norm{u}_{\ell^2(B_2)} \right).
    \end{equation}
    \end{enumerate}
\end{theorem}

The deduction of the Three Balls Inequality \eqref{TCT:4.2eq2} from \eqref{TCT:4.1eq1} is rather standard, we include
it here for sake of completeness:

\begin{proof}[Proof of $(1)$ implies $(2)$ in Theorem \ref{TBT:4.2}]
Let $h_0, \delta_0 \in (0,1)$, $c_1,c_2>0$ and $\tau_0>1$ such that for all $\tau\in(\tau_0,\delta_0h^{-1})$, $h\in(0,h_0)$
be given by Theorem \ref{TBT:4.2}. Let $\alpha\in(0,1)$ and $u:h\mathcal{V}\to\R$, with $H^L_{h\Gamma} u=0$ in $B_4$.
When $u=0$, \eqref{TCT:4.2eq2} is trivival so we assume $u\not=0$.
Take $\tau_*$ such that $e^{c_1\tau_*}\norm{u}_{\ell^2(B_{\frac{1}{2}})}=e^{-c_2\tau_*}\norm{u}_{\ell^2(B_2)}$.
Note that this implies that
\begin{equation*}
    e^{c_1\tau_*}\norm{u}_{\ell^2(B_{\frac{1}{2}})}=e^{-c_2\tau_*}\norm{u}_{\ell^2(B_2)}=\norm{u}_{\ell^2(B_{\frac{1}{2}})}^{\frac{c_2}{c_1+c_2}}\norm{u}_{\ell^2(B_2)}^{\frac{c_1}{c_1+c_2}}.
\end{equation*}
We now consider three different cases.
    \begin{enumerate}
        \item If $\tau_*\leq\tau_0$, then we have that
    \begin{eqnarray*}
        \norm{u}_{\ell^2(B_1)}&\leq& \norm{u}_{\ell^2(B_{2})}^\alpha\norm{u}_{\ell^2(B_2)}^{1-\alpha}
        =e^{(c_1+c_2)\alpha\tau_*}\norm{u}_{\ell^2(B_{1/2})}^\alpha\norm{u}_{\ell^2(B_2)}^{1-\alpha}\\
        &\leq& e^{(c_1+c_2)\alpha\tau_0}\norm{u}_{\ell^2(B_{1/2})}^\alpha\norm{u}_{\ell^2(B_2)}^{1-\alpha}
    \end{eqnarray*}
    which implies \eqref{TCT:4.2eq2} provided we take $D\geq e^{(c_1+c_2)\alpha\tau_0}$.
    
    \item If $\tau_*\in(\tau_0,\delta_0h^{-1})$, we write \eqref{TCT:4.1eq1} with $\tau=\tau_*$ and obtain
     \begin{eqnarray*} 
        \norm{u}_{\ell^2(B_1)} &\leq& C\left(e^{c_1\tau_*}\norm{u}_{\ell^2(B_{1/2})}+e^{-c_2\tau_*}\norm{u}_{\ell^2(B_2)}\right)\\
        &\leq&2C\norm{u}_{\ell^2(B_{\frac{1}{2}})}^{\frac{c_2}{c_1+c_2}}\norm{u}_{\ell^2(B_2)}^{\frac{c_1}{c_1+c_2}}
  \end{eqnarray*}
        which implies \eqref{TCT:4.2eq2} provided we take $D\geq 2C$.

    \item If $\tau_*>\delta_0h^{-1}$ we obtain that
\begin{equation*}
    e^{c_1h^{-1}\frac{\delta_0}{2}} \norm{u}_{\ell^2(B_\frac{1}{2})} \leq e^{c_1\tau_*} \norm{u}_{\ell^2(B_\frac{1}{2})}= e^{-c_2\tau_*}\norm{u}_{\ell^2(B_2)} \leq e^{-c_2h^{-1}\frac{\delta_0}{2}} \norm{u}_{\ell^2(B_2)}.
\end{equation*}
     Next, since we can assume that $\tau_0<h^{-1}\delta_0/2$ we can take $\tau=h^{-1}\dfrac{\delta_0}{2}$
     in \eqref{TCT:4.1eq1} to obtain
    \begin{align*}
         \norm{u}_{\ell^2(B_1)}&\leq C\left(e^{c_1h^{-1}\frac{\delta_0}{2}} \norm{u}_{\ell^2(B_\frac{1}{2})}+e^{-c_2h^{-1}\frac{\delta_0}{2}} \norm{u}_{\ell^2(B_2)}\right) \\
         &\leq 2Ce^{-c_2h^{-1}\frac{\delta_0}{2}} \norm{u}_{\ell^2(B_2)}.
    \end{align*}
    \end{enumerate}
    Combining the three cases we obtain \eqref{TCT:4.2eq2} with $D=\max(2C,e^{(c_1+c_2)\alpha\tau_0})$,
    $\alpha=\dfrac{c_2}{c_1+c_2}$ and $c_0=c_2\dfrac{\delta_0}{2}$.
\end{proof}

In order to prove Theorem \ref{TBT:4.2} for a 1-point periodic graph we need the following lemma, 

    \begin{lemma}[Caccioppoli]\label{lemma4}
        Let $u:h\mathcal{V}\mapsto\R$ be a weak solution of 
        \begin{align*}
            h^{-2}\sum_{j=1}^k &D^+_jD^-_ju(x)+h^{-1}\sum_{j=1}^kB_j(x)D_j^+u(x) \\
            &+\sum_{e\in\mathcal{P}_L} V_e(x+he)u(x+he)  =0,
        \end{align*}
        in the sense that $u\in H^1_{\text{loc},h}(h\mathcal{V})$ and for all $v\in H^1(h\mathcal{V})$ such that $\supp v$ is bounded, we have 
        \begin{align*}
            \sum_{x\in h\mathcal{V}}& h^{-2}\sum_{i,j=1}^k D^+_iu(x)D^+_jv(x)+\sum_{x\in h\mathcal{V}}h^{-1}\sum_{j=1}^kB_j(x)D_j^+u(x)v(x) \nonumber\\ 
            &+\sum_{x\in h\mathcal{V}}\sum_{e\in\mathcal{P}_L} V_e(x+he)u(x+he) v(x) =0. 
        \end{align*}
        Let $0<10h<r_1<r_1+100h<r_2$. Then there exists a constant $C>1$ depending on $r_1,r_2,\norm{V_e}_{L^\infty(h\mathcal{V})}$ and $\norm{B_j}_{L^\infty(h\mathcal{V})}$ such that
        \begin{equation*}
            \sum_{j=1}^k \norm{h^{-1}D_j^+u}^2_{\ell^2(B_{r_1})}\leq C\norm{u}_{\ell^2(B_{r_2})}^2.
        \end{equation*}
    \end{lemma}

    \begin{proof}
        Let $\eta:h\mathcal{V}\mapsto\R$ be a cut-off function which is equal to one on $B_{r_1}$ and vanishes outside outside $B_{r_2}$, and let 
        $$
        M=\max\{\|V_e\|_{L^\infty},\|B_j\|_{L^\infty}^2\}.
        $$ 
        From the Mean Value Theorem, we also have
        $$
        |D_j^+\eta(x)|=|\eta(x+he_j)-\eta(x)|\leq h\max_j\|e_j\|\max_{B_{r_2}}|\nabla\eta|:= C_\eta h.
        $$
        Let 
        $$
        A_{u,v}(x)=h^{-2}\sum_{j=1}^k D^+_iu(x)D^+_jv(x),
        $$
        and 
        \begin{align*}
            B_{u,v}(x)=&h^{-1}\sum_{j=1}^kB_j(x)D_j^+u(x)v(x) \\
            &+\sum_{e\in\mathcal{P}_L} V_e(x+he)u(x+he)  v(x).
        \end{align*}
        
        We may now take $v=(u\eta^2)(x)$ in the definition of the weak solution that we write as
        \begin{equation}\label{aux:4.4}
          \sum_{x\in h\mathcal{V}}\left(A_{u,u\eta^2}(x)+B_{u,u\eta^2}(x)\right)=0.
        \end{equation}
        As
        $$
        D_j^+[u\eta^2](x)=D_j^+[u](x)\eta^2(x)+u(x+he_j)D_j^+[\eta^2](x)
        $$
        we have
        \begin{eqnarray}
        \sum_{x\in h\mathcal{V}}A_{u,u\eta^2}(x)&:=&\sum_{x\in h\mathcal{V}} h^{-2}\sum_{j=1}^k D_j^+[u](x)D_j^+[u\eta^2](x)\notag\\
        &=&\sum_{j=1}^k\norm{h^{-1}\eta D_j^+u}^2_{\ell^2(h\mathcal{V})} \label{ineqcacci3} \\
        &&+h^{-2}\sum_{x\in h\mathcal{V}} \sum_{j=1}^k D_j^+[u](x)u(x+he_j)D_j^+[\eta^2](x).\label{ineqcacci}
        \end{eqnarray}
        We now bound the second term \eqref{ineqcacci} from below. We first write
$$
        \begin{aligned}
  D_j^+&[u](x)u(x+he_j)D_j^+[\eta^2](x)\\
  =&D_j^+[u](x)u(x+he_j)D_j^+[\eta](x) \bigl(\eta(x+he_j)+\eta(x)\bigr)\\
  =&2D_j^+[u](x)u(x+he_j)\eta(x)D_j^+[\eta](x) +D_j^+[u](x)u(x+he_j)D_j^+[\eta](x)^2\\
  \geq&-\frac{1}{2}D_j^+[u](x)^2\eta(x)^2-2u(x+he_j)^2D_j^+[\eta](x)^2\\
  &-\frac{1}{2}u(x+he_j)^2D_j^+[\eta](x)^2
  -\frac{1}{2}u(x)^2D_j^+[\eta](x)^2
        \end{aligned}
        $$
        where we have used $2ab\geq-\dfrac{a^2}{2}-2b^2$ for the first term and $ab\geq-\dfrac{a^2}{2}-\dfrac{b^2}{2}$
        for the second one after arranging factors properly.
        Notice that the 3 last terms are zero if $x\notin B_{r_2}$.
  Summing over $j=1,\ldots, k$ and then over $x\in h\mathcal{V}$, leads to     
        \begin{align*}
            \sum_{x\in h\mathcal{V}} h^{-2}\sum_{j=1}^k &D_j^+[u](x)u(x+he_j)D_j^+[\eta^2](x) \\
            &\geq -\frac{1}{2}\sum_{x\in h\mathcal{V}}\sum_{j=1}^k \left(\eta^2(x)h^{-2}\left(D_j^+u(x)\right)^2\right) \\
            & \qquad\qquad- 3k\left(\sum_{x\in B_{r_2}}|u(x)|^2\right)\sup_j\left[h^{-1}\left(D_j^+\eta(x)\right)\right]^2\\
            &\geq
            -\frac{1}{2}\sum_{j=1}^k\norm{h^{-1}\eta D_j^+u}^2_{\ell^2(h\mathcal{V})}
            -3kC_\eta^2\norm{u}^2_{\ell^2(B_{r_2})}.
        \end{align*}
Inserting this into \eqref{ineqcacci} gives
$$
\sum_{x\in h\mathcal{V}}A_{u,u\eta^2}(x)\geq 
\frac{1}{2}\sum_{j=1}^k\norm{h^{-1}\eta D_j^+u}^2_{\ell^2(h\mathcal{V})}
            -3kC_\eta^2\norm{u}^2_{\ell^2(B_{r_2})}.
$$
Using Young's inequality, we bound the contribution of $B_{u,u\eta^2}$ and get
$$
\sum_{x\in h\mathcal{V}}B_{u,u\eta^2}(x)\leq \frac{1}{4}\sum_{j=1}^k\norm{h^{-1}\eta D_j^+u}^2_{\ell^2(h\mathcal{V})}+M\left(|\mathcal{P}_L|+4k+1\right)\norm{u}^2_{\ell^2(B_{r_2})}.
$$
Now, going back to \eqref{aux:4.4}, we obtain
$$
\sum_{j=1}^k\norm{h^{-1}\eta D_j^+u}^2_{\ell^2(h\mathcal{V})}\leq C\norm{u}_{\ell^2(B_{r_2})}
$$
where $C$ is a constant that depends on $r_1,r_2$ and $M$. Finally, as $\eta=1$ on $B_{r_1}$ we have 
$$
\dst\sum_{j=1}^k\norm{h^{-1}D_j^+u}^2_{\ell^2(B_{r_1})}\leq \sum_{j=1}^k\norm{h^{-1}\eta D_j^+u}^2_{\ell^2(h\mathcal{V})}\leq C\norm{u}_{\ell^2(B_{r_2})}
$$
which is the expected inequality.            
    \end{proof}
    
    We now prove the first part of Theorem \ref{TBT:4.2} using a cut-off argument. 
 
    \begin{proof}[Proof of Theorem \ref{TBT:4.2} (1)]
      
    Let $u:h\mathcal{V}\mapsto\R$ such that $H^L_{h\Gamma}u(x)=0$ for all $x\in B_4$. We set $$
    M:= \max\left\{\norm{V_e}_{L^\infty},\|B_j\|_{L^\infty}\right\}<+\infty.
    $$
    Fix $\eps>0$ to be small enough and assume that $h_0>0$ is sufficiently small to have $(L+1)h_0\kappa<\eps$ with $\kappa=\max\|e_i\|$ and $h_0<\sqrt{M}e^{\tau(\ffi(3/2)-\ffi(1/4))}$. We consider the function $w(x)=\theta(x)u(x)$, with $0\leq\theta(x)\leq1$ a $\mathcal{C}^\infty(\R^d)$ cut-off function defined as
    \begin{equation*}
        \theta(x)=\left\{\begin{array}{cc}
             0, &\mbox{when }|x|_\Gamma<\frac{1}{4}+2\eps\mbox{ or }|x|_\Gamma>\frac{7}{4}-2\eps,\\[6pt]
             1, &\mbox{when }\frac{1}{3}-2\eps\leq |x|_\Gamma\leq\frac{3}{2}+2\eps.
        \end{array} \right.
    \end{equation*}
    We assume that $\eps$ is small enough to have 
    $$
    \dfrac{1}{4}+2\eps<\dfrac{1}{3}-2\eps<\dfrac{1}{2}<1<\dfrac{3}{2}+2\eps<\dfrac{7}{4}-2\eps<2.
    $$

         Now using the Carleman Estimate \eqref{carleman:eq1bis} on $w$ and the triangle inequality, we get that for all $\tau\in(\tau_0,\delta_0h^{-1})$
    \begin{multline*}
        \tau^3\norm{e^{\phi_\tau} w}^2+\tau\norm{h^{-1}e^{\phi_\tau}D_sw}^2\lesssim 
        \norm{e^{\phi_\tau} h^{-2}\Delta_{h\Gamma}w}^2  \\
\lesssim    \norm{e^{\phi_\tau} H^L_{h\Gamma}w}^2    +\sum_{e\in\mathcal{P}_L}
\norm{e^{\phi_\tau} V_e(\cdot+he)w(\cdot+he)}^2    + 
\norm{h^{-1}e^{\phi_\tau} \sum_{j=1}^kB_j(\cdot)D_s^jw(\cdot)}^2\\
\lesssim    \norm{e^{\phi_\tau} H^L_{h\Gamma}w}^2    +\sum_{e\in\mathcal{P}_L}
\norm{e^{\phi_\tau} w(\cdot+he)}  ^2+
\norm{h^{-1}e^{\phi_\tau}D_sw}^2
    \end{multline*}
since $\|V_e\|_\infty,\|B_j\|_\infty\lesssim 1$. 
Now, using the same argument as in the first step of the proof of our Carleman Inequality,
$\norm{e^{\phi_\tau} w(\cdot+he)}  ^2\lesssim\norm{e^{\phi_\tau} w(\cdot)}  ^2$.
We can thus absorb the two last terms on the left hand side into the corresponding terms on the right hand side
by taking $\tau$ large enough. We then obtain
  \begin{equation}\label{carl:1}
        \tau^3\norm{e^{\phi_\tau} w}^2+\tau\norm{h^{-1}e^{\phi_\tau}D_sw}^2
\lesssim    \norm{e^{\phi_\tau} H^L_{h\Gamma}w}^2.
\end{equation}

Now we further lower bound the left hand side.
First, we can discard the second term which is no longer useful. Next, since $w=u$ on $B_1\setminus B_{1/2}$
and since $\ffi$ is decreasing, we obtain
$$
        e^{2\tau\ffi(1)}\norm{u}^2_{\ell^2(B_{1}\setminus B_{1/2})}\leq \norm{e^{\phi_\tau} w}^2  .
$$
As $\tau^3\geq 1$, \eqref{carl:1} further reduces to
    \begin{equation}
        \label{eq:carl2}
e^{2\tau\ffi(1)}\norm{ u}^2_{\ell^2(B_{1}\setminus B_{1/2})}
\lesssim    \norm{e^{\phi_\tau} H^L_{h\Gamma}w}^2.
    \end{equation}

The next step consists in upper-bounding the right hand side. For this, we note that, as $H^L_{h\Gamma}u(x)=0$, we can then write
    \begin{align*}
       H^L_{h\Gamma}w(x)
        =&h^{-2}\sum_{j=1}^k\Delta_{j}\theta(x)u(x-he_j)\\
        &+ \sum_{j=1}^k h^{-2}D_j^+\theta(x)D_j^+u(x)+\sum_{j=1}^k h^{-2}D_j^+\theta(x)D_j^-u(x)\bigr) \\
        &+ h^{-1}\sum_{j=1}^kB_j(x)D_j^+\theta(x)u(x+he_j)-h^{-1}\sum_{j=1}^kB_j(x)D_j^-\theta(x)u(x-he_j)\\
        &+\sum_{e\in \mathcal{P}_L} \bigl(\theta(x+he)-\theta(x)\bigr)V_e(x+he)u(x+he) \\
        =&T_1u(x)+T_2^+u(x)+T_2^-u(x)+T_3^+u(x)+T_3^-u(x)+T_4u(x).
     \end{align*}
It then follows from \eqref{eq:carl2} that
\begin{eqnarray}
   \norm{e^{\phi_\tau} H^L_{h\Gamma}w}^2
&\lesssim& \norm{e^{\phi_\tau}T_1u}^2+\norm{e^{\phi_\tau}T_2^+u}^2+\norm{e^{\phi_\tau}T_2^-u}^2\notag\\
&&+\norm{e^{\phi_\tau}T_3^+u}^2+\norm{e^{\phi_\tau}T_3^-u}^2+\norm{e^{\phi_\tau}T_4u}^2.    \label{eq:carl3}
\end{eqnarray}
We will now bound each term on the right hand side.

Let us start with $T_1$. Define 
    $$
    \s=\left(B_{\frac{1}{3}-\eps}\setminus B_{\frac{1}{4}+\eps} \right)\cup\left( B_{\frac{7}{4}-\eps}\setminus B_{\frac{3}{2}+\eps}\right)
    $$
    and notice that $\Delta_j\theta(x)=D_j^\pm\theta(x)=0$ when $x\notin\s$
    and that, from the Mean Value Theorem, $h^{-1}|D_j^{\pm}\theta(x)|,h^{-2}|\Delta_j\theta(x)|\lesssim1$.
Then
    \begin{eqnarray}
 \norm{e^{\phi_\tau} T_1 u}^2
 &=&\sum_{x\in h\mathcal{V}} \abs{e^{\phi_\tau(x)}\sum_{j=1}^kh^{-2}\Delta_{j}\theta(x)u(x-he_j)}^2\notag\\
&\lesssim& \sum_{x\in h\mathcal{V}\cap\s} \abs{e^{\phi_\tau(x)}\sum_{j=1}^k|u(x-he_j)|}^2\notag\\
&\lesssim&e^{2\tau\varphi(1/4)}\norm{u}^2_{B_{\frac{1}{3}}\setminus B_{\frac{1}{4}}}
+e^{2\tau\varphi(3/2)}\norm{u}^2_{B_{\frac{7}{4}}\setminus B_{\frac{3}{2}}} \notag\\
        &\lesssim& e^{2\tau\varphi(1/4)}\norm{u}^2_{B_\frac{1}{2}}+e^{2\tau\varphi(3/2)}\norm{u}^2_{B_2} \label{eq:p3b2}
    \end{eqnarray}
where we used again the monotonicity properties of $\ffi$ in the third line.

We now bound $T_2^\pm$. 
First 
    \begin{eqnarray*}
    \norm{e^{\phi_\tau} T_2^\pm u(x)}^2
    &=&\sum_{x\in h\mathcal{V}} \abs{\sum_{j=1}^k h^{-2}e^{\phi_\tau(x)}\left(D_j^+\theta(x)D_j^\pm u(x)\right)}^2\\
    &\lesssim& \sum_{x\in h\mathcal{V}\cap \s} \abs{\sum_{j=1}^k h^{-1}e^{\phi_\tau(x)}\left|D_j^\pm u(x)\right|}^2\\        
    &\lesssim& e^{2\tau\varphi(1/4)}\norm{ h^{-1}\sum_{j=1}^kD_j^\pm u}^2_{B_{\frac{1}{3}}\setminus B_{\frac{1}{4}}} \\
    &&  \qquad  +e^{2\tau\varphi(3/2)}\norm{h^{-1}\sum_{j=1}^kD_j^\pm u}^2_{B_{\frac{7}{4}} \setminus B_{\frac{3}{2}}}.\\
    \end{eqnarray*}
    Using our Caccioppoli estimate given in Lemma \ref{lemma4}, we obtain
\begin{align*}
    \norm{ h^{-1}\sum_{j=1}^kD_j^\pm u}^2_{B_{\frac{1}{3}}\setminus B_{\frac{1}{4}}}\leq&
    \norm{ h^{-1}\sum_{j=1}^kD_j^\pm u}^2_{B_{\frac{1}{3}}}\lesssim \norm{u}^2_{B_{\frac{1}{2}}},\\
    \norm{h^{-1}\sum_{j=1}^kD_j^\pm u}^2_{B_{\frac{7}{4}} \setminus B_{\frac{3}{2}}}
    \leq&\norm{h^{-1}\sum_{j=1}^kD_j^\pm u}^2_{B_{\frac{7}{4}}}
    \lesssim \norm{u}^2_{B_2}.
\end{align*}
Thus ultimately leads to
\begin{equation}
    \label{eq:p3b3}
\norm{e^{\phi_\tau} T_2^\pm u(x)}^2\lesssim e^{2\tau\varphi(1/4)}\norm{u}^2_{B_\frac{1}{2}}+e^{2\tau\varphi(3/2)}\norm{u}^2_{B_2}.
\end{equation}

We now turn to $T_3^\pm$ and $T_4$. The computations are similar for all three of them
and rely on the bound $\norm{V_e}_{L^\infty},\|B_j\|_{L^\infty}\leq M$ and on the fact that for each $e\in\mathcal{P}_L$,
    $$
    |\theta(x-he)-\theta(x)| \lesssim h\mathbf{1}_{\s}.
    $$
It follows that
\begin{eqnarray}
   \|e^{\phi_\tau}T_4u\|^2&\lesssim&
   \sum_{e\in\mathcal{P}_L}\|e^{\phi_\tau}\bigl(\theta(x+he)-\theta(x)\bigr)V_e(x+he)u(x+he)\|^2\notag\\
   &\lesssim& h^2M^2\|e^{\phi_\tau}u\mathbf{1}_{\s}\|^2\notag\\
   &\lesssim& e^{2\tau\varphi(1/4)}\norm{u}^2_{B_\frac{1}{2}}+e^{2\tau\varphi(3/2)}\norm{u}^2_{B_2} \label{eq:boundt4}
\end{eqnarray}
where we have concluded in the same way as the last line of \eqref{eq:p3b2}.

The bound for $T_3^\pm$ is obtained the same way:
\begin{equation}
    \|e^{\phi_\tau}T_3^\pm u\|^2\lesssim e^{2\tau\varphi(1/4)}\norm{u}^2_{B_\frac{1}{2}}+e^{2\tau\varphi(3/2)}\norm{u}^2_{B_2}.\label{eq:boundt3}
\end{equation}

\smallskip

We can now conclude by introducing  \eqref{eq:p3b2}-\eqref{eq:p3b3}-\eqref{eq:boundt4}-\eqref{eq:boundt3} 
into \eqref{eq:carl3}:
$$
\norm{e^{\phi_\tau} H^L_{h\Gamma}w}^2
\lesssim e^{2\tau\varphi(1/4)}\norm{u}^2_{B_\frac{1}{2}}+e^{2\tau\varphi(3/2)}\norm{u}^2_{B_2}
$$
which with \eqref{eq:carl2} gives
$$
e^{2\tau \varphi(1)}\norm{u}^2_{B_{1} \setminus B_{\frac{1}{2}}}
\leq C\bigl(e^{2\tau\varphi(1/4)}\norm{u}^2_{B_\frac{1}{2}}+e^{2\tau\varphi(3/2)}\norm{u}^2_{B_2}\bigr)
$$
for some $C>0$. We rewrite this as
$$
\norm{u}^2_{B_{1} \setminus B_{\frac{1}{2}}}
\leq C\bigl(e^{2\tau\bigl(\varphi(1/4)-\varphi(1)\bigr)}\norm{u}^2_{B_\frac{1}{2}}
+e^{2\tau\bigl(\varphi(3/2)-\varphi(1)\bigr)}\norm{u}^2_{B_2}\bigr).
$$
Finally, we add $\norm{u}^2_{B_{\frac{1}{2}}}$ on both sides and notice that
$$
c_1:=2\bigl(\varphi(1/4)-\varphi(1)\bigr)>0 \quad\mbox{ and } \quad c_2:=2\bigl(\varphi(1)-\varphi(3/2)\bigr)>0
$$ 
to obtain
$$
 \norm{u}^2_{B_{1}}
\leq (1+C)\bigl(e^{c_1\tau}\norm{u}^2_{B_\frac{1}{2}}
+e^{-c_2\tau}\norm{u}^2_{B_2}\bigr)
$$
as claimed.
\end{proof}

\section{Extending the family of graphs}

\subsection{Hexagonal Lattice}

Our goal now is to extend Theorem \ref{TBT:4.2} to a larger family of graphs. We start with the extension to the Hexagonal Lattice, a famous periodic graph that is not a 1-point periodic graph.
It is a part of a slightly larger family of periodic graphs that we describe now.

\begin{center}
\tikzset{
    my hex/.style={regular polygon, regular polygon sides=6, draw, inner sep=0pt, outer sep=0pt, minimum size=1cm},
    my circ/.style={draw, circle, color=blue, fill=blue, inner sep=0pt, minimum size=0.75mm},
    my tri/.style={regular polygon, regular polygon sides=3, draw, inner sep=0pt, outer sep=0pt, minimum size=1mm, color=red, fill=red},
    my squa/.style={regular polygon, regular polygon sides=4, draw, inner sep=0pt, outer sep=0pt, minimum size=1mm, color=green, fill=green}
}
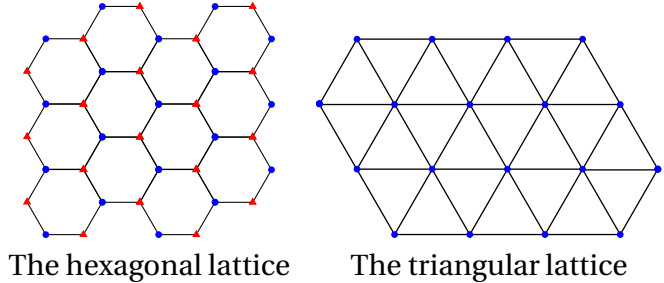
\begin{figure}[ht]
\begin{tabular}{cc}
\begin{tikzpicture}
 \foreach \i in {0,...,1} 
    \foreach \j in {0,...,2} {
            \node(h)[my hex] at (1.5*\i,0.87*\j) {};
      \foreach \t in {1,3,5} \node[my circ] at ($(h)+({(\t-1)*60}:.5)$){};
      \foreach \t in {2,4,6} \node[my tri] at ($(h)+({(\t-1)*60}:.5)$){};
            \node(h)[my hex] at (1.5*\i+0.75,0.87*\j+0.43) {};
       \foreach \t in {1,3,5} \node[my circ] at ($(h)+({(\t-1)*60}:.5)$){};
      \foreach \t in {2,4,6} \node[my tri] at ($(h)+({(\t-1)*60}:.5)$){};
            }
\end{tikzpicture} & 
\begin{tikzpicture}
    \foreach \j in {0,...,1} {
    \draw [line width=.5pt] (0.5,\j*1.73) -- (3.5,\j*1.73);
    \draw [line width=.5pt] (0,\j*1.73+0.87) -- (3,\j*1.73+0.87);
    \draw [line width=.5pt] (0.5+\j,0) -- (2+\j,2.6);
    \draw [line width=.5pt] (0+3.5*\j,0.87+\j*0.87) -- (1+1.5*\j,2.6-2.6*\j);
    \draw [line width=.5pt] (-0.5+3.5*\j,1.73+\j*0.87) -- (0.5+3.5*\j,0.87*\j);
    \draw [line width=.5pt] (-0.5+4*\j,1.73-\j*1.73) -- (0+4*\j,2.6-1.73*\j);
    \draw [line width=.5pt] (-0.5+3.5*\j,1.73-0.87*\j) -- (0.5+\j*3.5,1.73-0.87*\j);
    }
    \foreach \i in {0,...,2} {
    \draw [line width=.5pt] (\i,2.6) -- (\i+1.5,0);
    }
    \foreach \j in {0,...,1}
    \foreach \i in {0,...,3} {
    \node(h)[my circ] at (\i+0.5,\j*1.73) {};
    \node(h)[my circ] at (\i,\j*1.73+0.87) {};
    \node(h)[my circ] at (4-\j*4.5,0.87+\j*0.87) {}; 
    }
\end{tikzpicture} \\
     The hexagonal lattice & The triangular lattice
\end{tabular}
\caption{The hexagonal lattice (2 points) and the triangular lattice (1 point) associated}
\end{figure}
\end{center}
\smallskip

\begin{definition}[Hexagonal type periodic graphs]\label{defhex:1}
We consider a lattice $\mathcal{L}$ in $\R^d$ and two points $p_1,p_2\in\R^d$
such that $p_2-p_2\notin \mathcal{L}$.
The periodic graph $h\Gamma=\{h\mathcal{L},h\mathcal{V},h\mathcal{E}\}$
with mesh size $h$ is of hexagonal type if
\begin{enumerate}
    \item the vertex set $h\mathcal{V}$ is of the form $h\mathcal{V}=h\mathcal{V}_1\cup h\mathcal{V}_2$ with $h\mathcal{V}_j=hp_j+h\mathcal{L} $;

    \item there exist vectors $\mathbf{e_m}\in\mathcal{L}$, $j=1,\ldots,k$ such that
    \begin{enumerate}
        \item if $x\in h\mathcal{V}_1$ then 
    $(x,y)\in h\mathcal{E}$ if and only if 
    $y=x+h(p_2-p_1+\mathbf{e_m})$ for some $j=1,\ldots,k$, in particular $y\in h\mathcal{V}_2$;

    \item if $x\in h\mathcal{V}_2$ then 
    $(x,y)\in h\mathcal{E}$ if and only if 
    $y=x-h(p_2-p_1+\mathbf{e_m})$ for some $j=1,\ldots,k$, in particular $y\in h\mathcal{V}_1$.
    \end{enumerate}
    \item the graph is connected.    
\end{enumerate}
\end{definition} 
The third condition is just the fact that the graph is not oriented. Note that each vertex has degree $k$
and the connectedness condition implies that $k\geq d$.

For the remaining of this subsection, we will assume that $h\Gamma=\{h\mathcal{L},h\mathcal{V},h\mathcal{E}\}$ is of hexagonal type. Then, every function $g:h\mathcal{V}\mapsto \C$ may be identified
with the function $\bar g=(g_1,g_2):h\mathcal{L}\mapsto \C^2$ where $g_j(\cdot)=g(hp_j+\cdot)$.
Note that
$$
\|G\|_{\ell^2(h\mathcal{L})}^2:=\|g_1\|_{\ell^2(h\mathcal{L})}^2+\|g_2\|_{\ell^2(h\mathcal{L})}^2
=\|g\|_{\ell^2(h\mathcal{V})}^2.
$$

\begin{center}
    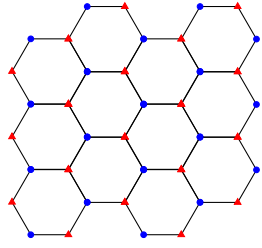
\begin{figure}[ht]
    \tikzset{
    my hex/.style={regular polygon, regular polygon sides=6, draw, inner sep=0pt, outer sep=0pt, minimum size=1cm},
    my circ/.style={draw, circle, color=blue, fill=blue, inner sep=0pt, minimum size=0.75mm},
    my tri/.style={regular polygon, regular polygon sides=3, draw, inner sep=0pt, outer sep=0pt, minimum size=1mm, color=red, fill=red},
    my squa/.style={regular polygon, regular polygon sides=4, draw, inner sep=0pt, outer sep=0pt, minimum size=1mm, color=green, fill=green}
}
\begin{tikzpicture}
 \foreach \i in {0,...,1} 
    \foreach \j in {0,...,2} {
            \node(h)[my hex] at (1.5*\i,0.87*\j) {};
      \foreach \t in {1,3,5} \node[my circ] at ($(h)+({(\t-1)*60}:.5)$){};
      \foreach \t in {2,4,6} \node[my tri] at ($(h)+({(\t-1)*60}:.5)$){};
            \node(h)[my hex] at (1.5*\i+0.75,0.87*\j+0.43) {};
       \foreach \t in {1,3,5} \node[my circ] at ($(h)+({(\t-1)*60}:.5)$){};
      \foreach \t in {2,4,6} \node[my tri] at ($(h)+({(\t-1)*60}:.5)$){};
            }
\end{tikzpicture}
\caption{The standard Hexagonal lattice, $g_1$ would be the restriction of $g$ to the blue (circle) points) and $g_2$ the restriction to the red (triangle) points}
\end{figure}
\end{center}

%
%

For the remaining of Section 5, we study the operator 
$$
H_h f(x)=h^{-2}\Delta_{\mathcal{H}}f(x)+V(x)f(x).
$$

\begin{theorem}[Three Balls Inequality]\label{TBT:hex}
Let $\mathcal{H}$ be an hexagonal type periodic graph:
    \begin{enumerate}
    \item There exist constants $h_0, \delta_0 \in (0,1)$, $c_1,c_2>0$ and $\tau_0>1$ such that for all $\tau\in(\tau_0,\delta_0h^{-1})$, $h\in(0,h_0)$ and $u:h\mathcal{V}\mapsto\R$ with $H_hu=0$ on $B_4$ we have
    \begin{equation}
        \norm{u}_{\ell^2(B_1)} \leq C\left(e^{c_1\tau}\norm{u}_{\ell^2(B_{1/2})}+e^{-c_2\tau}\norm{u}_{\ell^2(B_2)}\right).\label{eq:51}
    \end{equation}
\item There exist $\alpha \in (0,1)$, $c_0>0$, $h_0\in(0,1)$ and $C>1$ such that for
$h\in(0,h_0)$ and $u:h\mathcal{V}\to\R$ with $H_h u=0$ in $B_4$ we have 
    \begin{equation}\label{TCT:hex}
        \norm{u}_{\ell^2(B_1)} \leq D\left(\norm{u}_{\ell^2(B_{1/2})}^\alpha\norm{u}_{\ell^2(B_2)}^{1-\alpha} + e^{-c_0h^{-1}}\norm{u}_{\ell^2(B_2)} \right).
    \end{equation}
    \end{enumerate}
\end{theorem}



The idea of the proof is that a periodic graph of hexagonal type has two one-point
periodic graph hidden in it. The vertices are $h\mathcal{V}_j$ and $(x,y)$ is an edge
if $y$ is a neighbor of a neighbor of $x$ in the original graph. The key is then to
write $H_hf=0$ as $H_{j,h}f_j=0$ where $H_{j,h}$ are operators of the form
\eqref{eq:formhextrig} on the hidden one-point periodic graphs.

\begin{center}
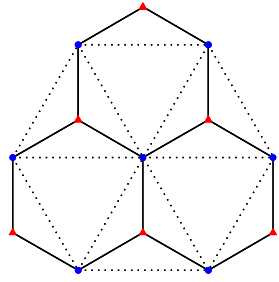
\begin{figure}[ht]
    \tikzset{
    my circ/.style={draw, circle, color=blue, fill=blue, inner sep=0pt, minimum size=0.75mm},
    my tri/.style={regular polygon, regular polygon sides=3, draw, inner sep=0pt, outer sep=0pt, minimum size=1mm, color=red, fill=red},
}
    \centering
\begin{tikzpicture}
\foreach \i in {0,...,2}
\foreach \j in {0,...,1} {
\draw [line width=.5pt] (\i*1.73,1) -- (\i*1.73,2);
\draw [line width=.5pt] (\j*1.73+0.87,2.5) -- (\j*1.73+0.87,3.5);
\draw [line width=.5pt] (0.87+\j*0.87,3.5+\j*0.5) -- (1.73+\j*0.87,4-0.5*\j);
\draw [line width=.5pt] (\j*1.73,2) -- (\j*1.73+0.87,2.5);
\draw [line width=.5pt] (\j*1.73,1) -- (\j*1.73+0.87,0.5);
\draw [line width=.5pt] (\j*1.73+0.87,2.5) -- (\j*1.73+1.73,2);
\draw [line width=.5pt] (\j*1.73+0.87,0.5) -- (\j*1.73+1.73,1);
\draw [line width=.5pt,dotted] (\j*1.73,2) -- (\j*1.73+1.73,2);
\draw [line width=.5pt,dotted] (0.87,0.5+\j*3) -- (2.6,0.5+\j*3);
\draw [line width=.5pt,dotted] (0+\j*3.47,2) -- (0.87+\j*1.73,3.5-3*\j);
\draw [line width=.5pt,dotted] (2.6*\j,2+1.5*\j) -- (0.87+2.6*\j,0.5+1.5*\j);
\draw [line width=.5pt,dotted] (0.87,0.5+\j*3) -- (2.6,3.5-3*\j);
}
\foreach \i in {0,...,2}
\foreach \j in {0,...,1} {
\node(h)[my circ] at (0.87,0.5+3*\j) {};
\node(h)[my circ] at (0.87+1.73,0.5+3*\j) {};
\node(h)[my tri] at (\j*1.73+0.87,2.5) {};
\node(h)[my circ] at (\i*1.73,2) {};
\node(h)[my tri] at (\i*1.73,1) {};
\node(h)[my tri] at (1.73,4) {};
}
\end{tikzpicture}
 \caption{The triangular lattice hidden in the Hexagonal graph}
\end{figure}
\end{center}

\begin{proof}
    Let $f\,:\mathcal{V}\to\C$ be such that $H_h f=0$ in $B_4(\mathcal{H})$. 
    Identify $f,V$ with $\bar f=(f_1,f_2),\bar V=(V_1,V_2):\  h\mathcal{L}\mapsto\C^2$.
    Then, for $x\in h\mathcal{L}$,
    \begin{eqnarray}
        H_h f(hp_1+x)&=&\frac{1}{h^{2}}\left(\sum_{m=1}^kf_2(x+h\mathbf{e}_m)
    -\bigl(k-h^2V_1(x)\bigr)f_1(x)\right)\label{eq:hfhex}\\
     H_h f(hp_2+x)&=&\frac{1}{h^{2}}\left(\sum_{m=1}^kf_1(x-h\mathbf{e}_m)
    -\bigl(k-h^2V_2(x)\bigr)f_2(x)\right).\notag
    \end{eqnarray}

Assuming that $h$ is small enough to have $h^2\|V\|_\infty<k$, we can rewrite $H_hf=0$ as
    \begin{equation}
    \label{eq:f1f2}\begin{aligned}
        f_1(x)=&\frac{1}{k-h^2V_1(x)}\left(\sum_{m=1}^kf_2(x+h\mathbf{e}_m)\right)\\
f_2(x)=&\frac{1}{k-h^2V_2(x)}\left(\sum_{j=1}^kf_1(x-h\mathbf{e}_j)\right)\\
=&\left(\frac{1}{k-h^2V_2(x)}-\frac{1}{k}\right)\sum_{m=1}^kf_1(x-h\mathbf{e}_m)+
\frac{1}{k}\sum_{j=1}^kf_1(x-h\mathbf{e}_j)
\end{aligned}
    \end{equation}

We now define a new 1 point periodic graph as follows $h\mathcal{T}=(h\mathcal{L},h\mathcal{V}',h\mathcal{E}')$
with $(x,y)\in\mathcal{E}'$ if $y=x+h(\mathbf{e}_m-\mathbf{e}_n)$, $m\not=n\in\{1,\ldots,k\}$.
Reordering the vectors $\mathbf{e}_m-\mathbf{e}_n$, $m<n$
as $\mathbf{\eta}_m$, $m=1,\ldots,K=\dfrac{k(k-1)}{2}$, we obtain that $h\mathcal{T}$
is a 1-point periodic graph with edges $(x,y)$, $y=x\pm h\mathbf{\eta}_j$, $j=1,\ldots,K$.
Note that to each $\eta_j$ there corresponds a unique $m(j),n(j)$ such that $\mathbf{\eta}_j=
\mathbf{e}_{m(j)}-\mathbf{e}_{n(j)}$.
The Laplace operator for $h\mathcal{T}$ is given by
$$
\Delta_\mathcal{T}\ffi(x)=\sum_{j=1}^K\bigl(\ffi(x+h\mathbf{\eta}_j)+\ffi(x-h\mathbf{\eta}_j)\bigr)
-2K\ffi(x).
$$
Now, injecting the expression of $f_2$ from \eqref{eq:f1f2} into \eqref{eq:hfhex},
\begin{eqnarray*}
0= H_h f(hp_1+x)&=&\frac{1}{h^{2}}\left(\sum_{m=1}^k\frac{1}{k}\sum_{n=1}^kf_1(x+h\mathbf{e}_m+h\mathbf{e}_n)
    -kf_1(x)\right)\\
    &&+\sum_{m=1}^k\frac{V_2(x+h\mathbf{e}_m)}{k-h^2V_2(x+h\mathbf{e}_m)}\sum_{n=1}^kf_1(x+h\mathbf{e}_m-h\mathbf{e}_n)
    +V_1(x)f_1(x)\\
    &=&\frac{1}{k h^{2}}\Delta_\mathcal{T}f_1(x)+\tilde V(x)f_1(x)\\
    &+&\sum_{j=1}^K\bigl(V_j^+(x+h\mathbf{\eta_j})f_1(x+h\mathbf{\eta_j})+V_j^-(x-h\mathbf{\eta_j})f_1(x-h\mathbf{\eta_j})\bigr) \\
    &:=& \frac{1}{k}H_{1,h}f_1(x)
\end{eqnarray*}
where $\tilde V=\dst V_1(x)+\sum_{m=1}^k\frac{V_2(x+h\mathbf{e}_m)}{k-h^2V_2(x+h\mathbf{e}_m)}$
and
each $V_j^\pm$ is a term of the form $\dfrac{V_2(x+h\mathbf{e}_m)}{k-h^2V_2(x+h\mathbf{e}_m)}$.
Those terms are thus bounded independently of $h\leq h_0$, provided $h_0$ is small enough.

Similarly, we obtain that 
    \begin{equation*}
        0=H_h f(hp_2+x)= \frac{1}{k}H_{2,h}f_2(x)
    \end{equation*}
with $H_{2,h}$ an operator of a similar form to $H_{1,h}$.
Both operators are of the form given by \eqref{eq:formhextrig} on the one-point
periodic graph $\mathcal{T}$ and each
$f_j: \ h\mathcal{L}\mapsto\R$ is a function on $\mathcal{T}$, for which
$H_{j,h}f_j=0$ on $x\in B_4(\mathcal{T})$.
We thus know from Theorem \ref{TBT:4.2} that there exist constants $h_0, \delta_0 \in (0,1)$, $c_1,c_2>0$ and $\tau_0>1$ such that for all $\tau\in(\tau_0,\delta_0h^{-1})$, $h\in(0,h_0)$ we have
    \begin{equation*}
        \norm{f_j}_{\ell^2(B_1)} \lesssim \left(e^{c_1\tau}\norm{f_j}_{\ell^2(B_{1/2})}+e^{-c_2\tau}\norm{f_j}_{\ell^2(B_2)}\right)
    \end{equation*}
   or equivalently 
   \begin{equation}\label{hex:tbt:1}
       \norm{f_j}^2_{\ell^2(B_1)} \lesssim \left(e^{2c_1\tau}\norm{f_j}^2_{\ell^2(B_{1/2})}+e^{-2c_2\tau}\norm{f_j}^2_{\ell^2(B_2)}\right).
   \end{equation}
   
    Finally, since for all $r\geq0$, 
    $$
    \norm{f}^2_{B_r}=\norm{f_1}^2_{B_r}+\norm{f_2}^2_{B_r},
    $$
    summing \eqref{hex:tbt:1} for $j=1,2$, we obtain
    $$
    \norm{f}_{\ell^2(B_1)}^2 \lesssim \left(e^{2c_1\tau}\norm{f}_{\ell^2(B_{1/2})}^2+e^{-2c_2\tau}\norm{f}_{\ell^2(B_2)}^2\right)
    $$
    which is equivalent to \eqref{eq:51}.

The proof that this implies \eqref{TCT:hex} is standard and the argument
was already given for Theorem \ref{TBT:4.2}.
\end{proof}


\subsection{Star periodic graphs}

We can extend Theorem \ref{TBT:hex} to star periodic graphs.

\begin{definition}\label{star:1} 
    We say that a periodic graph $h\Gamma=(h\mathcal{V},h\mathcal{E},h\mathcal{L})$ is a star periodic graph, if for every $x\in h\mathcal{V}_j$, $j=2,\dots, s$, then for every $y$ such that $(x,y)\in h\mathcal{E}$, then $y\in h\mathcal{V}_1$.
\end{definition}

In other words, a start periodic graph is obtained from a one-point periodic graph
by adding layers $\mathcal{V}_j=p_j+\mathcal{L}$, $j=2,\ldots,s$ that only have neighbors
in the first graph.

\begin{center}
    \tikzset{
    my circ/.style={draw, circle, color=blue, fill=blue, inner sep=0pt, minimum size=0.75mm},
    my tri/.style={regular polygon, regular polygon sides=3, draw, inner sep=0pt, outer sep=0pt, minimum size=1mm, color=red, fill=red},
}
\begin{figure}[ht]
    \begin{tabular}{cc}
        \begin{tikzpicture}[scale=0.5]
        \foreach \l in {0,...,1}
        \foreach \u in {0,...,1}
        \foreach \i in {0,...,2}
        \foreach \j in {0,...,3}
        \foreach \k in {0,...,4}
         {
        \draw [line width=.5pt] (\l*3,\j*1.73) -- (\l*3+1,\j*1.73);
        \draw [line width=.5pt] (1.5+\l*3,\j*1.73+0.865) -- (\l*3+2.5,\j*1.73+0.865);
        \draw [line width=.5pt] (1.5*\k,1.73*\l+0.865*\k) -- (1.5*\k-0.5,0.865+1.73*\l+\k*0.865);
        \draw [line width=.5pt] (5.5-1.5*\k,0.865+0.865*\k) -- (6-1.5*\k,1.73+0.865*\k);
        \draw [line width=.5pt] (-0.5+1.5*\u+4.5*\l,0.865-0.865*\u+4.32*\l) -- (1.5*\u+4.5*\l,1.73-0.865*\u+4.32*\l);
        \draw [line width=.5pt] (2.5+1.5*\i-\l*3,0.865+0.865*\i+3.46*\l) -- (3+1.5*\i-3*\l,0.865*\i+3.46*\l);
        \draw [line width=.5pt] (4+1.5*\l-1.5*\j,0.865*\j+2.6*\l) -- (4.5+1.5*\l-1.5*\j,0.865+0.865*\j+2.6*\l);
        \draw [line width=.5pt] (\l*3+\u*1.5,0.865*\u) -- (\l*3+\u*1.5,5.19+0.865*\u);
        }
        \foreach \j in {0,...,3}
        \foreach \l in {0,...,2}
        \foreach \u in {0,...,1}
        \foreach \i in {0,...,1} {
        \node(h)[my circ] at (6,1.73+\l*1.73) {};
        \node(h)[my circ] at (3*\i+1.5*\u,\j*1.73+0.865*\u) {};
        \node(h)[my tri] at (-0.5,0.865+\l*1.73) {};
        \node(h)[my tri] at (1+3*\i+1.5*\u,1.73*\j+0.865*\u) {};
        }         
\end{tikzpicture}
&  \begin{tikzpicture}[scale=0.5]
\foreach \i in {0,...,6}
\foreach \j in {0,...,6} {
\draw [line width=.5pt] (\i,0) -- (\i,6);
\draw [line width=.5pt] (0,\j) -- (6,\j);
}
\foreach \i in {0,...,6}
\foreach \j in {0,...,6}
\foreach \t in {0,...,5} {
\node(h)[my circ] at (\i,\j) {}; 
\node(h)[my tri] at (\t+0.5,\j) {};
\node(h)[my tri] at (\i,\t+0.5) {}; 
}

\end{tikzpicture}
\\
 a) Hexagonal star graph        & b) Subdivision of the square lattice
    \end{tabular}
    \caption{Some modification of the hexagonal lattice and subdivision of 1 point periodic graph are star periodic graphs}
    \end{figure}
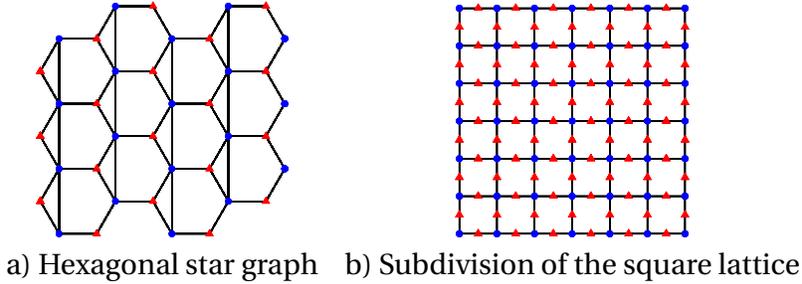
\end{center}

As for graphs of hexagonal type, a function $f$ on $h\mathcal{V}$ can then be seen as a vector
valued function $(f_1,\ldots,f_s)$ on $h\mathcal{L}$. Now, if $h^{-2}\Delta f+Vf=0$
then $f_2,\ldots,f_s$ can be expressed in terms of $f_1$ only. In a similar way to the previous
section, one can then show that $H_{1,h}f_1=0$ where $H_{1,h}$ is an operator of the form
\eqref{eq:formhextrig} on a one-point periodic graph whose vertices are $h\mathcal{L}$.
From there, one then proceeds as in the proof of Theorem \ref{TBT:hex}:

\begin{theorem}[Three Balls Inequality]\label{TBT:hex2}
Let $\mathcal{H}$ be a star periodic graph, then the Three Balls Inequality, Theorem \ref{TBT:hex},
is also valid on $\mathcal{H}$.
%
\end{theorem}


\section{Data availability}
No data has been generated or analysed during this study.

\section{Funding and/or Conflicts of interests/Competing interests}

The authors have no relevant financial or non-financial interests to disclose.

\smallskip

Y. Bourroux has benefited from state
support managed by the Agence Nationale de la
Recherche (French National Research Agency)
under reference ANR-20-SFRI-0001

\smallskip

Y. Bourroux and P. Jaming are partially supported by the French National Research Agency (ANR) under contract number ANR-24-CE40-5470.

\smallskip

A. Fern\'andez-Bertolin is partially supported by the Basque Government through the project IT1615-22, by the Spanish Agencia Estatal de Investigaci\'on through the project PID2021-122156NB-I00 funded by MICIU/AEI/10.13039/501100011033 and by Euskampus, through the LTC Sarea Transmath.

\bibliographystyle{alpha}

\begin{thebibliography}{}

\end{thebibliography}


\begin{thebibliography}{FBRRS21}

\bibitem[AIM15]{ando_spectral_2015}
K. Ando, H. Isozaki \& H. Morioka.
\newblock {\em Spectral Properties of Schrödinger Operators on Perturbed Lattices.}
\newblock Ann. Henri Poincaré {\bf 17} (2016), 2103--2171.

\bibitem[AIM18]{ando_inverse_2018}
K. Ando, H. Isozaki \& H. Morioka.
\newblock {\em Inverse Scattering for Schrödinger Operators on Perturbed Lattices.}
\newblock Ann. Henri Poincaré {\bf 19} (2018), 3397--3455.


\bibitem[BRCG23]{bou-rabee_unique_2023}
A. Bou-Rabee, W. Cooperman \& S. Ganguly.
\newblock {\em Unique continuation on planar graphs.}
\newblock arXiv:2309.13728.

\bibitem[CNG$^{+}$09]{CNGPNG09}
A.\,H. Castro Neto, F. Guinea, N.\,M.\,R. Peres, K.\,S. Novoselov \& A. Geim.
{\em The electronic properties of graphene.}
Rev. Modern Phys. {\bf 81} (2009), 109-–162.

\bibitem[FBR$^{+}$21]{fernandez-bertolin_discrete_2021}
A. Fernández-Bertolin, L. Roncal, A. Rüland \& D. Stan.
\newblock {\em Discrete Carleman estimates and three balls inequalities.}
\newblock Calc. Var. Partial Differ. Equ. {\bf 60} (2021), Paper No. 239. 

\bibitem[GM14]{guadie_three_2014}
M. Guadie \& E. Malinnikova.
\newblock {\em On Three Balls Theorem for Discrete Harmonic Functions}.
\newblock Comput. Methods Funct. Theory {\bf 14} (2014), 721--734.

\bibitem[HML$^{+}$10]{hamma_quantum_2010}
A. Hamma, F. Markopoulou, S. Lloyd, F. Caravelli, S.
  Severini \& K. Markström.
\newblock {\em Quantum {Bose}-{Hubbard} model with an evolving graph as a toy model
  for emergent spacetime.}
\newblock {\em Phys. Rev. D}, 81 (2010), 104032.

\bibitem[Ha02]{Ha02}
P. Harris.
Carbon Nano-Tubes and Related Structure,
{\it Cambridge University Press}, Cambridge, 2002.

\bibitem[KRS16]{koch_variable_2016}
H. Koch, A. Rüland \& W. Shi.
\newblock{\em The variable coefficient thin obstacle problem: {Carleman}
  inequalities.}
\newblock {\em Adv. Math.}, {\bf 301} (2016), 820--866.

\bibitem[KMS08]{konopka_quantum_2008}
T. Konopka, F. Markopoulou \& Simone Severini.
\newblock {\em Quantum graphity: {A} model of emergent locality.}
\newblock {\em Phys. Rev. D}, {\bf 77} (2008), 104029.

\bibitem[KM94]{korevaar_logarithmic_1994}
J.~Korevaar \& J.~L.~H. Meyers.
\newblock {\em Logarithmic {Convexity} for {Supremum} {Norms} of {Harmonic}
  {Functions}.}
\newblock Bull. Lond. Math. Soc. {\bf 26} (1994), 353--362.

\bibitem[KS14]{badanin_laplacians_2013}
E. Korotyaev \& N. Saburova.
\newblock {\em Schr\"odinger operators on periodic discrete graphs.}
\newblock J. Math. Anal. Appl. {\bf 420} (2014), 576–-611.

\bibitem[KS18]{KS}
E. Korotyaev \& N. Saburova.
\newblock {\em Laplacians on periodic discrete graphs with guides.}
J. Math. Anal. Appl. {\bf 455} (2017), 1444--1469

\bibitem[Ku97]{kurata_unique_1997}
K. Kurata.
\newblock {\em A {Unique} {Continuation} {Theorem} for the {Schrödinger} {Equation}
  with {Singular} {Magnetic} {Field}.}
\newblock {\em Proc. Am. Math. Soc.},
{\bf 125} (1997), 853--860.

\bibitem[NGM$^{+}$04]{NGMJZDGF04}
K.\,S. Novoselov, A.\,K. Geim, S.\,V. Morozov, D. Jiang, Y. Zhang, S.\,V. Dubonos, I.\,V. Grigorieva \& A.\,A. Firsov.
{\em Electric field effect in atomically thin carbon films.}
Science {\bf 306} (2004), 666-–669.

\bibitem[Sa24]{Sa24}
N. Saburova
{\em Asymptotic isospectrality of Schrödinger operators on periodic graphs.}
Anal. Math. Phys. {\bf 14} (2024), 74.

\bibitem[YLL$^{+}$22]{yan_periodic_2022}
K. Yan, Y. Liu, Y. Lin \& S. Ji.
\newblock{\em Periodic {Graph} {Transformers} for {Crystal} {Material} {Property}
  {Prediction}.}
\newblock {\em Advances in Neural Information Processing Systems},
{\bf 35} (2022), 15066--15080.


\end{thebibliography}

\end{document}